\documentclass[12pt]{amsart}

\usepackage[T1]{fontenc}
\usepackage[utf8]{inputenc}
\usepackage{lmodern}
\usepackage{amsmath,amssymb,amsthm,mathtools}
\usepackage{enumitem}
\usepackage{booktabs}
\usepackage{microtype}
\usepackage[a4paper,margin=1in]{geometry}
\usepackage{comment}
\usepackage{tikz}
\usetikzlibrary{arrows.meta, positioning, calc}
\usepackage{float}

\usepackage{xcolor}
\usepackage[
    colorlinks=true,
    linkcolor=blue!60!black,
    citecolor=blue!60!black,
    urlcolor=blue!60!black
]{hyperref}

\numberwithin{equation}{section}
\setlist[enumerate]{leftmargin=*,itemsep=2pt,topsep=4pt}
\setlist[itemize]{leftmargin=*,itemsep=2pt,topsep=4pt}
\newtheorem{theorem}{Theorem}[section]
\newtheorem*{theorem*}{Theorem}
\newtheorem{proposition}[theorem]{Proposition}
\newtheorem{lemma}[theorem]{Lemma}
\newtheorem{corollary}[theorem]{Corollary}
\theoremstyle{definition}
\newtheorem{definition}[theorem]{Definition}
\newtheorem{example}[theorem]{Example}
\theoremstyle{remark}
\newtheorem{remark}[theorem]{Remark}

\theoremstyle{plain}
\newtheorem{teoremaletra}{Theorem}

\newcommand{\Omn}{\mathcal{O}_{m,n}}

\newcommand{\KMS}{\mathrm{KMS}}
\newcommand{\C}{\mathbb C}

\newcommand{\R}{\mathbb R}
\newcommand{\T}{\mathbb T}

\newcommand{\id}{\operatorname{id}}

\newcommand{\Tr}{\operatorname{Tr}}
\newcommand{\Ad}{\operatorname{Ad}}

\title[KMS and ground states for $\Omn$]{KMS and ground states for the normalized \\ dynamics of the tame $C^*$-algebras $\mathcal{O}_{m,n}$}

\author{Paulo R.\ Pinto}
\address[Paulo R.\ Pinto]{Department of Mathematics and Center for
	Mathematical Analysis, Geometry and Dynamical Systems, CAMGSD,  Instituto Superior Técnico, University of Lisbon, Av. Rovisco Pais 1, 1049-001, Lisboa, Portugal.} 
\email{ppinto@math.tecnico.ulisboa.pt}

\author{Filipe Viseu}  
\address[Filipe Viseu]{Department of Mathematics and Center for
	Mathematical Analysis, Geometry and Dynamical Systems, CAMGSD, Instituto Superior Técnico, University of Lisbon, Av. Rovisco Pais 1, 1049-001, Lisboa, Portugal.}  
\email{filipe.viseu@tecnico.ulisboa.pt}  
\date{\today}

\subjclass[2020]{Primary 46L30; Secondary 46L55, 37B20, 82B10}
\keywords{KMS state: ground state: tame $C^*$-algebra; dynamical system of type 
$(m,n)$; partial isometry; full corner;
Fell bundles; polytope}

\begin{document}

\begin{abstract}
We study $\KMS$ and ground states of the tame $C^*$-algebras $\mathcal{O}_{m,n}$ of Ara, Exel and Katsura under the dynamics $\sigma_t(s_i)=m^{it}s_i$, $\sigma_t(t_j)=n^{it}t_j$. Normalized restriction to the full corner $A=q\mathcal{O}_{m,n}q$, generated by the partial isometries $v_{ij}=s_i^{*}t_j$ of common energy $\log(n/m)$, is at each inverse temperature $\beta$ an affine weak-$*$ homeomorphism of the two systems. Let $m,n\ge 2$. When $m=n$ the corner dynamics is trivial, so the $\KMS$ simplices are identified with trace simplices. When $m\neq n$, $\KMS$ states occur only at $\beta=1$, and there they are not unique. Conformal branching models surject the $\KMS_1$ simplex onto a transportation polytope of dimension $(m-1)(n-1)$, whose vertices lift to extremal, hence factorial, $\KMS_1$ states, the fibre over each vertex being itself the $\KMS_1$ simplex of a
quotient of $A$.
\end{abstract}

\maketitle

\tableofcontents

\section{Introduction}
\label{sec:introduction}
 
The $\KMS$ condition originates in the work of Kubo and of Martin and
Schwinger \cite{Kubo1957,MartinSchwinger1959} in quantum statistical
mechanics, and was formulated in operator-algebraic terms by Haag, Hugenholtz
and Winnink \cite{HHW1967}, who proposed it as the defining property of
equilibrium states. Since the 1970s, the classification of $\KMS$ states has
been carried out for a wide range of $C^*$-dynamical systems. In this text, we use
the standard formulation of Bratteli--Robinson \cite{BR2}, recalled at the
beginning of Section~\ref{sec:corner}. For the Cuntz and Cuntz--Krieger
algebras, the critical inverse temperatures and the corresponding equilibrium
states are closely tied to the number of generators, the underlying transition
matrix, and Perron--Frobenius data
\cite{Cuntz1977,CK1980,OP1978,EFW1984}. More general quasi-free, graph-algebra
and partial-dynamical-system versions of this theme have been developed in
\cite{ExelLaca2003,LacaNeshveyev2004,aHLRS2013,CarlsenLarsen2016}, and groupoid
and Fell-bundle formulations describe KMS states in terms of quasi-invariant
measures and measurable isotropy data \cite{Neshveyev2013,AfsarSims2021}.
 
The $C^*$-algebras $\Omn$ were introduced by Ara, Exel and Katsura in
their study of universal dynamical systems of type $(m,n)$ \cite{AEK2013}. They
are generated by two finite families of tame partial isometries
$s_1,\ldots,s_m$, and $t_1,\ldots,t_n$, with a common initial projection
$q=s_i^*s_i=t_j^*t_j$ and two decompositions of a common range projection
$
        p=\sum_{i=1}^m s_is_i^*=\sum_{j=1}^n t_jt_j^* .
$
The full defining relations, including the tameness condition, are also
recalled at the beginning of Section~\ref{sec:corner}.\footnote{Our indexing is chosen so that there are $m$ generators $s_i$ and
$n$ generators $t_j$; the opposite convention is used for example in \cite{AEK2013},
where the two families are labelled by $n$ and by $m$ respectively. See 
Remark~\ref{rem:mn} below.} These algebras lie at
the intersection of several themes in operator algebras, including Leavitt-type
relations, partial crossed products, Fell bundles, separated graphs \cite[Example~6.6]{AraLolk2018}, and interaction crossed
products \cite{Exel2018}. They are also far from tractable by classification
machinery: when $m,n\ge2$, the tame universal algebra is non-exact, whereas its
reduced regular counterpart $\mathcal O^r_{m,n}$ is exact but non-nuclear \cite{AEK2013}. Throughout we work with the full algebra. Unlike its reduced counterpart, it
is defined by a universal property, which we use repeatedly to construct
representations.

As these $C^*$-algebras may be viewed as a two-sided generalization of the Cuntz
algebras, there is a natural choice of dynamics, namely
$\sigma_t(s_i)=m^{it}s_i$ and $\sigma_t(t_j)=n^{it}t_j$, where each family of
generators is scaled according to its cardinality, as in the standard gauge
dynamics on the Cuntz algebra \cite{OP1978}. In the case of the Cuntz algebra
$\mathcal{O}_n$, it is known that this dynamics admits a unique $\KMS$ state,
at inverse temperature $\beta=1$ in the present normalization, while no ground
states exist. The question that motivates this paper is what happens when the
same type of dynamics is considered simultaneously on the two families of
generators of $\Omn$. As it turns out, the answer depends sharply on whether
$m=n$, and already this particular choice leads to phenomena quite different
from the rigid picture for the Cuntz algebras; we do not yet have a complete
description of the resulting equilibrium states. 

In \cite{AEK2013}, the $C^*$-algebras $\Omn$ are realized as partial crossed products $\Omn\cong C(\Omega^{u})\rtimes_{\theta^{u}} \mathbb F_{m+n}$ of a partial action $\theta^{u}$ of the free group $\mathbb F_{m+n}$ on the universal $(m,n)$--dynamical system. This suggests that the KMS states might be studied using the results of \cite{ExelLaca2003}. However, the classification theorem in \cite[Theorem~4.3]{ExelLaca2003} assumes that the partial action is orthogonal, whereas the partial action considered here clearly does not satisfy this condition; see \cite[Definition~4.2]{ExelLaca2003} and \cite[Section~2]{AEK2013}. One could alternatively try a groupoid approach. In this setting, by
\cite[Theorem~1.3]{Neshveyev2013}, the $\KMS_\beta$ states correspond to pairs
consisting of a probability measure on $\Omega^u$ that is quasi-invariant with
Radon--Nikodym cocycle $e^{-\beta c}$, together with an measurable
field of traces satisfying certain conditions. For our purposes this is more a reformulation
than a classification, as it does not produce the measure explicitly,
whereas the conformal models of Section~\ref{sec:regular-models} do, and it
leaves the remaining problem essentially untouched. As Neshveyev himself observes, this description is hardly a simplification
when the dynamics is complicated and the isotropy is large \cite[Introduction]{Neshveyev2013}. We therefore do not pursue the groupoid approach further, and work instead
directly with the algebra.

The organizing result of the paper is Theorem~\ref{thm:A}, proved in
Section~\ref{ss:corner-algebra}. The defining relations of $\Omn$ make $q$
equivalent to each of the projections $s_is_i^*$, so $q$ is full and $\Omn$ is a
matrix algebra over its corner. Writing
$
        A=q\Omn q,$ $
        v_{ij}=s_i^*t_j,$ and $
        \varepsilon=\log(n/m),
$
the corner is the universal unital $C^*$-algebra generated by a tame family of
$mn$ partial isometries subject to
\begin{equation}\tag{$*$}\label{eq:corner-relations-intro}
        \sum_{j=1}^n v_{ij}v_{ij}^*=q\quad(1\le i\le m),
        \qquad
        \sum_{i=1}^m v_{ij}^*v_{ij}=q\quad(1\le j\le n),
\end{equation}
and the normalized dynamics restricts to an action
$\sigma^A\colon\mathbb R\to\operatorname{Aut}(A)$ with
$\sigma^A_t(v_{ij})=e^{it\varepsilon}v_{ij}$. Two families carrying two energies
are thus replaced by a single family carrying one. Theorem~\ref{thm:A} below shows
that $\varphi\mapsto(1+m^{1-\beta})\varphi|_A$ is an affine weak-$*$
homeomorphism of $\KMS_\beta$ simplices at every inverse temperature, and that
restriction itself is one of ground-state spaces whenever $(m,n)\neq(1,1)$.

The rigid regimes are those in which $(A,\sigma^A)$ is trivial or classical
(Sections~\ref{ss:diagonal}--\ref{ss:ground}). In the diagonal case $m=n=:r\ge2$ we have $\varepsilon=0$, so the
corner $\KMS$ states are precisely the traces at every inverse
temperature; consequently the $\KMS_\beta$ simplex of $O_{r,r}$ is
canonically affinely homeomorphic to its trace simplex for every
$\beta\in\mathbb R$. The dynamics is inner, implemented by
$r^{it}p+q$, and explicit finite-dimensional representations settle
existence and non-uniqueness at every $\beta$. In the boundary cases
\eqref{eq:corner-relations-intro} degenerates: for $m=1$ into the Cuntz
relations, hence $A\cong\mathcal O_n$ and
$\mathcal O_{1,n}\cong M_2(\mathcal O_n)\cong\mathcal O_{n,1}$ \cite{AEK2013} with amplified
gauge dynamics; for $(m,n)=(1,1)$ into
$\mathcal O_{1,1}\cong M_2(C(\mathbb T))$ with trivial action. At zero
temperature and $m\neq n$, every ground state of the corner must vanish on $q$,
so there are none and the dynamics is not approximately inner
\cite[Theorem~2.3]{PowersSakai}; for $m=n$ the ground states are exactly those
supported on $A$.

Section~\ref{sec:asymmetric} treats $m,n\ge2$. For $m\neq n$ we have
$\varepsilon\neq0$, and we will show that KMS states can occur only at $\beta=1$. Existence is proved
in Section~\ref{sec:regular-models} by constructing \textit{conformal branching models}
from two finite measurable decompositions of a probability space, a construction
that applies equally when $m=n$. The first-level intersection data of such a
model are the entries of the 2-way transportation polytope
$$
\mathcal{P}:=\left\{ C=(c_{ij})_{ij}\in M_{m\times n}(\mathbb{R}):
        c_{ij}\ge0, \
        \sum_j c_{ij}=\frac1m, \
        \sum_i c_{ij}=\frac1n \right\},
$$
whose row and column constraints are precisely
\eqref{eq:corner-relations-intro}, rescaled by $m$ and $n$, and whose dimension
is $(m-1)(n-1)$. Distinct matrices of $\mathcal P$ are realized by distinct
models and separated by the values of the overlap projections $f_{ij}$ of the
ranges $s_is_i^*$ and $t_jt_j^*$; since these values are defined for every
$\KMS_1$ state, they yield an affine, weak-$*$ continuous surjection $\Phi$ from
the $\KMS_1$ simplex onto $\mathcal P$, and in particular non-uniqueness
(Section~\ref{sec:nonunique}).
Section~\ref{sec:fibres} shows that every vertex of $\mathcal P$ is the image of
an extremal, hence factorial, $\KMS_1$ state, and that distinct vertices give
pairwise disjoint states. Nevertheless, $\Phi$ is never injective: for $\max(m,n)\ge3$ the polytope
$\mathcal P$ has $mn$ facets and is therefore not a simplex, whereas the
$\KMS_1$ states form one; for $(m,n)=(2,2)$, two distinct states arising
from the finite-dimensional representations constructed before already have
the same image under $\Phi$. Writing $Z$ for the set of indices at which a vertex $V$ vanishes,
$\Phi^{-1}(V)$ is affinely homeomorphic to the $\KMS_1$ simplex of the quotient
of $A$ by the ideal generated by $\{v_{ij}:(i,j)\in Z\}$; when $m=n=r\ge2$ these
quotients are all isomorphic to $C^*(\mathbb F_r)$, so each of the $r!$ fibres
is identified with $\operatorname{Tr}(C^*(\mathbb F_r))$.

By Theorem~\ref{thm:A}, whether the states constructed here exhaust the
$\KMS_1$ simplex is the question of classifying the $\KMS_1$ states of
$(A,\sigma^A)$; it remains open, as does the structure of the fibres of $\Phi$
away from the vertices. A natural next step is the higher-level data attached to
words of increasing length in the $v_{ij}$, which may lead to a hierarchy of
transportation polytopes, as motivated in Example~\ref{ex:fibre}. Moreover, our arguments for the existence of the factorial $\mathrm{KMS}_1$ states in
Theorem~\ref{thm:E}(d) are essentially combinatorial, relying only on the
extremality properties of the vertices of the polytope $\mathcal P$. Therefore, we leave the determination of the von Neumann type of the corresponding factors for further investigation.

Finally, one more word on the choice of dynamics. For the general action
$\sigma_t(s_i)=e^{it\lambda_i}s_i$, $\sigma_t(t_j)=e^{it\kappa_j}t_j$, the
corner reduction of Theorem~\ref{thm:A} holds  with the weights
$d_i=e^{\lambda_i}$, the corner generators then having energies
$\kappa_j-\lambda_i$, and the constraint $\beta=1$ is replaced by the partition
identity \eqref{eq:partition}. Two points, however, are sensitive to the
normalization: the transportation
polytope of Section~\ref{sec:nonunique} is central only for uniform weights,
and for general margins it may be a simplex, so that the facet argument for the
non-injectivity of $\Phi$ breaks down; and the ground-state analysis of
Section~\ref{ss:ground} uses that all corner energies have the same sign, which
for general weights need not hold. We therefore work throughout with the normalized dynamics, which is where
our interest originated and where both results hold in full.

\subsection{Main results}
 
We now state our main results on the classification of KMS and ground states
for the normalized dynamics
\begin{equation}\label{eq:normalized-action}
\sigma_t(s_i)=m^{it}s_i,\qquad
\sigma_t(t_j)=n^{it}t_j
\end{equation}
on $\Omn$. Theorem~\ref{thm:A} is the reduction to the corner
$q\Omn q$ on which everything else rests; from it the rigid regimes
follow according to: whether the corner energy $\varepsilon=\log(n/m)$ vanishes, namely the
diagonal case $m=n\ge2$ in Theorem~\ref{thm:B}; the boundary cases in which $m$
or $n$ equals $1$ in Theorem~\ref{thm:C}; and the ground states in all cases in
Theorem~\ref{thm:D}. The transportation polytope and its consequences is the subject of
Theorem~\ref{thm:E} and occupies Section~\ref{sec:asymmetric}.
Table~\ref{table-results} summarizes the picture, and the pointers in each item
indicate where the corresponding statement is proved.
 
\begin{teoremaletra}[Reduction to the corner]\label{thm:A}
Let $m,n\ge1$, let $\sigma$ denote the normalized dynamics
\eqref{eq:normalized-action}, and put
$A=q\Omn q$, $v_{ij}=s_i^{*}t_j,$ and $\varepsilon=\log(n/m).
$
\begin{enumerate}[label=\textup{(\alph*)}]
\item The family $\{v_{ij}\}$ is tame, generates $A$, and
satisfies the corner relations
\[
        \sum_{j=1}^n v_{ij}v_{ij}^{*}=q\quad(1\le i\le m),
        \qquad
        \sum_{i=1}^m v_{ij}^{*}v_{ij}=q\quad(1\le j\le n).
\]
Moreover $A$ is universal for them: if $\{w_{ij}\}$ is a tame family satisfying
the same relations in a unital $C^*$-algebra $B$, there is a unique unital
$*$-homomorphism $A\to B$ with $v_{ij}\mapsto w_{ij}$
(Propositions~\ref{prop:corner-generators} and~\ref{prop:Auniversal}).

\item The dynamics $\sigma$ restricts to a strongly
continuous action $\sigma^A$ on $A$, and all generators have the same energy,
$
        \sigma^A_t(v_{ij})=e^{it\varepsilon}v_{ij}
$
(Proposition~\ref{prop:corner-generators}).

\item $q$ is a full projection, and
$x\mapsto(u_a^{*}xu_b)_{a,b}$, with $u_0=q$ and $u_i=s_i$, is a
$*$-isomorphism $\Omn\cong M_{m+1}(A)$ carrying $\sigma$ to $\sigma^A$ twisted
by the weights $(1,m,\dots,m)$; interchanging the two families gives
$\Omn\cong M_{n+1}(A)$ with the weights $(1,n,\dots,n)$
(Proposition~\ref{prop:matrix-picture}).

\item For every $\beta\in\mathbb R$, writing
$Z_\beta=1+m^{1-\beta}$, the map $\varphi\mapsto Z_\beta\varphi|_A$ is an affine
weak-$*$ homeomorphism from the $\KMS_\beta$ simplex of $(\Omn,\sigma)$ onto
that of $(A,\sigma^A)$, with inverse
\[
        \psi\longmapsto
        \frac1{Z_\beta}\Big(\psi(q\cdot q)
        +m^{-\beta}\sum_{i=1}^m\psi(s_i^{*}\cdot s_i)\Big)
\]
(Theorem~\ref{thm:corner-reduction}).

\item If $(m,n)\neq(1,1)$, restriction to $A$ is
likewise an affine homeomorphism from the ground-state space of $(\Omn,\sigma)$
onto that of $(A,\sigma^A)$, with inverse $\psi\mapsto\psi(q\cdot q)$
(Proposition~\ref{prop:corner-ground}).
\end{enumerate}
\end{teoremaletra}
 
\begin{teoremaletra}[The diagonal case]\label{thm:B}
Let $r\ge2$ and let $\sigma$ denote the normalized dynamics
$\sigma_t(s_i)=r^{it}s_i$, $\sigma_t(t_j)=r^{it}t_j$ on $\mathcal O_{r,r}$.
Here $\varepsilon=0$, so the corner dynamics $\sigma^A$ is trivial and the KMS
condition in $A$ is the trace identity at every inverse temperature.
Consequently, for every $\beta\in\mathbb R$ the simplex of $\KMS_\beta$ states
is canonically affinely homeomorphic, in the weak-$*$ topology, to the infinite-dimensional trace
simplex $\operatorname{Tr}(\mathcal O_{r,r})$
(Corollary~\ref{cor:corner-diagonal} and Proposition~\ref{prop:free-group-traces}), the homeomorphism being the Gibbs
correspondence
$$
        \tau\longmapsto
        \frac{r+1}{1+r^{1-\beta}}\tau\big(e^{-\beta h}\cdot\big),
        \qquad h=(\log r)p
$$
(Corollary~\ref{cor:gibbs}). In particular $\KMS_\beta$ states exist and are
non-unique at every real inverse temperature
(Corollaries~\ref{cor:findimtraces} and~\ref{cor:nonunique}).
\end{teoremaletra}
 
\begin{teoremaletra}[The boundary cases]\label{thm:C}
Let $\sigma$ denote the normalized dynamics \eqref{eq:normalized-action}.
\begin{enumerate}[label=\textup{(\alph*)}]
\item Let $n\ge2$. For $m=1$ the corner relations of Theorem~\ref{thm:A}(a)
degenerate into the Cuntz relations, so that $A\cong\mathcal O_n$ and
$
        \mathcal O_{1,n}\cong M_2(\mathcal O_n)\cong\mathcal O_{n,1},
$
with $\sigma$ becoming $\operatorname{id}_{M_2}\otimes\gamma$, where
$\gamma_t(w_j)=n^{it}w_j$ is the standard gauge dynamics on $\mathcal O_n$.
Consequently each of these $C^*$-algebras admits a unique $\KMS_1$ state and no
$\KMS_\beta$ state for $\beta\neq1$ (Corollary~\ref{cor:corner-boundary}).
\item For $(m,n)=(1,1)$ the corner is generated by the single unitary
$v_{11}=s^{*}t$, hence $A\cong C(\mathbb T)$ and
$\mathcal O_{1,1}\cong M_2(C(\mathbb T))$; the dynamics is trivial, and for
every $\beta\in\mathbb R$ the $\KMS_\beta$ states are precisely the tracial
states (Proposition~\ref{prop:O11}).
\end{enumerate}
\end{teoremaletra}
 
\begin{teoremaletra}[Ground states]\label{thm:D}
Let $m,n\ge1$ and let $\sigma$ denote the normalized dynamics
\eqref{eq:normalized-action} on $\Omn$
(Corollary~\ref{cor:corner-ground-cases}).
\begin{enumerate}[label=\textup{(\alph*)}]
\item If $m\neq n$, there are no ground states; hence $\sigma$ is not
approximately inner, by \cite[Theorem~2.3]{PowersSakai}.
\item If $m=n=r\ge2$, the ground states are precisely the states vanishing on
$p$, and restriction $\varphi\mapsto\varphi|_{q\mathcal O_{r,r}q}$ is an affine
homeomorphism onto the state space of the corner $q\mathcal O_{r,r}q$.
\item If $m=n=1$, every state is a ground state.
\end{enumerate}
\end{teoremaletra}

\begin{teoremaletra}[The polytope of $\KMS_1$ states]\label{thm:E}
Let $m,n\ge2$ and let $\sigma$ denote the normalized dynamics
\eqref{eq:normalized-action} on $\Omn$.
\begin{enumerate}[label=\textup{(\alph*)}]
\item If $m\neq n$, finite-temperature $\KMS_\beta$ states exist only at
$\beta=1$ (Corollary~\ref{cor:beta-one}); in all cases $\Omn$ admits a
$\KMS_1$ state (Theorem~\ref{thm:existence-asymmetric}).
\item The first-level intersection data define an affine, weak-$*$ continuous
surjection $\Phi$ from the $\KMS_1$ simplex onto the transportation polytope
$\mathcal P$, whose dimension is $(m-1)(n-1)$ and whose row and column
constraints are given by the two corner relations of Theorem~\ref{thm:A}(a)
(Proposition~\ref{prop:Phi}); in particular the $\KMS_1$ states are not unique
(Corollary~\ref{cor:nonuniqueness-asymmetric}).
\item The map $\Phi$ is not injective  (Proposition~\ref{prop:notinjective}).
\item Every vertex $V$ of $\mathcal P$ is the image of an extremal, hence
factorial, $\KMS_1$ state, and distinct vertices give pairwise disjoint states
(Theorem~\ref{thm:vertices}). Moreover, writing $Z$ for the set of indices at
which $V$ vanishes and $J_Z$ for the closed ideal of $A$ generated by
$\{v_{ij}:(i,j)\in Z\}$, the fibre $\Phi^{-1}(V)$ is affinely homeomorphic to
the $\KMS_1$ simplex of $A/J_Z$ (Theorem~\ref{thm:fibres-vertex}).
\item If $m=n=r$, the vertices of $\mathcal P$ are precisely the matrices
$r^{-1}P_\varpi$ with $P_\varpi$ an $r\times r$ permutation matrix. For each of
them $A/J_Z\cong C^*(\mathbb F_r)$, so the corresponding fibre is affinely
homeomorphic to $\operatorname{Tr}\left(C^*(\mathbb F_r)\right)$
(Proposition~\ref{prop:diagonal-fibres}).
\end{enumerate}
\end{teoremaletra}

\begin{table}[H]
\centering
\caption{KMS and ground states for the normalized dynamics in
\eqref{eq:normalized-action}.}
\label{table-results}
\small
\begin{tabular}{@{}llll@{}}
\toprule
Values of $m,n$ & Inverse temperatures & Behaviour & Ground states \\
\midrule
$m,n\ge2$, $m\ne n$ & $\beta=1$ only & existence; non-unique & none \\
$m=n=r\ge2$ & every $\beta\in\R$ & simplex $\cong\operatorname{Tr}(\mathcal O_{r,r})$ & states on $q\mathcal O_{r,r}q$ \\
$(1,n)$ or $(n,1)$, $n\ge2$ & $\beta=1$ only & $M_2(\mathcal O_n)$; unique & none \\
$(1,1)$ & every $\beta\in\R$ & $M_2(C(\mathbb T))$; tracial & all states \\
\bottomrule
\end{tabular}
\end{table}

\section{Preliminaries and first obstruction}
\label{sec:corner}

Before turning to the classification of $\KMS$ states for the normalized
dynamics on $\Omn$, we briefly recall the standard definitions
and results we shall use.  We refer the reader to
\cite[Section~5.3]{BR2} for a comprehensive treatment of the KMS theory.

Let $(B,\sigma)$ be a $C^*$-dynamical system, consisting of a
$C^*$-algebra $B$ together with a strongly continuous one-parameter group
of automorphisms $\sigma:\mathbb{R}\to\operatorname{Aut}(B)$. An element
$a\in B$ is said to be \emph{analytic} for $\sigma$ if the map
$t\mapsto\sigma_t(a)$ extends to an entire function $z\mapsto\sigma_z(a)$
taking values in $B$.

Given $\beta\in\mathbb{R}$, a state $\varphi$ on $B$ is called a
\emph{$\KMS_\beta$ state (with respect to $\sigma$)} if
$
    \varphi(ab)=\varphi\left(b\sigma_{i\beta}(a)\right)
$
for all analytic elements $a,b$ on a dense, $\sigma$-invariant $*$-subalgebra of entire
analytic elements. It is a general fact that, if $B$ has an identity, the
set of $\KMS_\beta$ states is a weak-$*$ compact simplex; we denote it by
$K_\beta(B,\sigma)$. We write also $\operatorname{Tr}(B)$ for the trace
simplex of $B$ in the weak-$*$ topology. At $\beta=0$ we use the convention given by the $\KMS$ identity, so that the $\KMS_0$ condition is precisely the trace
identity.

In what follows, let $m,n\ge1$. The tame algebra $\Omn$ is
generated by partial isometries $s_1,\ldots,s_m$ and $t_1,\ldots,t_n$,
subject to the following relations
\begin{align}
        s_i^*s_k&=0 &&(i\ne k),\label{rel:sorth}\\
        t_j^*t_\ell&=0 &&(j\ne \ell),\label{rel:torth}\\
        s_i^*s_i&=t_j^*t_j=:q,\label{rel:common-domain}\\
        \sum_{i=1}^m s_is_i^*&=\sum_{j=1}^n t_jt_j^*=:p,\label{rel:common-range}\\
        pq&=0,\qquad p+q=1.\label{rel:pq}
\end{align}
and required in addition to form a tame family in the sense of Definition~\ref{def:tame}.

\begin{remark}\label{rem:mn}
We generate $\Omn$ by $m$ partial isometries $s_i$ and $n$ partial
isometries $t_j$, whereas \cite{AEK2013} uses the reverse labelling and
assumes $m\le n$ throughout. Since \eqref{rel:sorth}--\eqref{rel:pq} are
invariant under interchanging the two families together with $m$ and $n$, the
resulting algebras coincide. Thus results quoted from \cite{AEK2013} are to be
read with $s$ and $t$ exchanged. Throughout, we also assume that the indices $i,j$ range over $1\le i\le m$ and $1\le j\le n.$
\end{remark}

\begin{definition}\label{def:tame}
Let $B$ be a unital $C^*$-algebra and let
$S_1,\ldots,S_m,T_1,\ldots,T_n\in B$ be partial isometries. The family
$\{S_i,T_j\}$ is called \emph{tame} if every element of the multiplicative
semigroup generated by
$S_1,\ldots,S_m,T_1,\ldots,T_n,S_1^*,\ldots,S_m^*,T_1^*,\ldots,T_n^*$ is
again a partial isometry.
\end{definition}

The tame algebra $\Omn$ is characterized by the following
universal property in \cite{AEK2013}.

\begin{proposition}\label{prop:universal}
Let $B$ be a unital $C^*$-algebra containing projections $P,Q$ and a tame
family of partial isometries $S_1,\ldots,S_m,T_1,\ldots,T_n$ satisfying
relations \eqref{rel:sorth}--\eqref{rel:pq}, with $P,Q,S_i,T_j$ in place
of $p,q,s_i,t_j$. Then there exists a unique unital $*$-homomorphism
$\pi:\Omn\to B$ with $\pi(s_i)=S_i$ and $\pi(t_j)=T_j$ for all
$i,j$. Conversely, the canonical generators of $\Omn$ satisfy
\eqref{rel:sorth}--\eqref{rel:pq} and form a tame family. Thus
$\Omn$ is the universal unital $C^*$-algebra for tame families
satisfying \eqref{rel:sorth}--\eqref{rel:pq}.
\end{proposition}

\begin{remark}\label{rem:inverse-semigroup}
Note that tameness is equivalent to requiring that the generators and
their adjoints generate an inverse semigroup of partial isometries; in
particular, the idempotents of this semigroup commute. We shall use this
concretely: since the range projections $s_is_i^*$ and $t_jt_j^*$ commute,
$s_is_i^*t_jt_j^*=t_jt_j^*s_is_i^*$ is a projection. When positivity
should be visible without first invoking commutativity, we write the same
projection as $t_jt_j^*s_is_i^*t_jt_j^*$.
\end{remark}

Consider the more general action given on the generators of
$\Omn$ by
\begin{equation}\label{eq:general-action}
        \sigma_t(s_i)=e^{it\lambda_i}s_i,
        \qquad
        \sigma_t(t_j)=e^{it\kappa_j}t_j ,
\end{equation}
where $\lambda_i,\kappa_j\in\mathbb R$ are fixed weights. It is
straightforward to check that $\sigma$ extends to a well-defined strongly
continuous action. Indeed, for each real $t$, using that
$|e^{it\lambda_i}|=|e^{it\kappa_j}|=1$, one checks that the elements
$e^{it\lambda_i}s_i$ and $e^{it\kappa_j}t_j$ are partial isometries
satisfying \eqref{rel:sorth}--\eqref{rel:pq} and that they form a tame
set. Hence, by Proposition~\ref{prop:universal}, there exists a unique
unital $*$-homomorphism $\sigma_t$ with $\sigma_t(s_i)=e^{it\lambda_i}s_i$
and $\sigma_t(t_j)=e^{it\kappa_j}t_j$. Uniqueness implies
that $\sigma_s\circ\sigma_t=\sigma_{s+t}$ and $\sigma_0=\id$; in particular, each $\sigma_t$ is invertible with inverse $\sigma_{-t}$, hence a $*$-automorphism. Strong
continuity of $t\mapsto\sigma_t$ follows from continuity on the dense
$*$-subalgebra generated by the canonical generators, and extends to all
of $\Omn$ since each $\sigma_t$ is isometric.

\begin{lemma}\label{lem:basic-kms}
Let $v$ be analytic for $\sigma$, and suppose that $\sigma_t(v)=e^{itE}v$
for all real $t$. If $\varphi$ is a $\KMS_\beta$ state, then
$\varphi(vv^*)=e^{-\beta E}\varphi(v^*v)$.
\end{lemma}

\begin{proof}
The hypothesis implies that the entire analytic continuation of the orbit
of $v$ is $\sigma_z(v)=e^{izE}v$, for $z\in\C$, because the right-hand
side is an entire function and agrees with the given real orbit. Applying
the KMS identity to the pair $v,v^*$ gives
$
\varphi(vv^*)=\varphi\left(v^*\sigma_{i\beta}(v)\right)
             =e^{-\beta E}\varphi(v^*v).
$
\end{proof}

The previous lemma is a standard computation. Applying it to the
generators of $\Omn$ and using the defining relations yields an
identity that restricts the values of the inverse temperature parameter
for which a KMS state can exist.

\begin{theorem}\label{thm:partition-equation}
If $\varphi$ is a $\KMS_\beta$ state for \eqref{eq:general-action}, then
\begin{equation}\label{eq:partition}
        \sum_{i=1}^m e^{-\beta\lambda_i}=\sum_{j=1}^n e^{-\beta\kappa_j}.
\end{equation}
\end{theorem}

\begin{proof}
For each $i$, Lemma~\ref{lem:basic-kms} applied to $v=s_i$ gives
$\varphi(s_is_i^*)=e^{-\beta\lambda_i}\varphi(s_i^*s_i)
=e^{-\beta\lambda_i}\varphi(q)$. Similarly
$\varphi(t_jt_j^*)=e^{-\beta\kappa_j}\varphi(q)$. Summing and using the
common range projection relation \eqref{rel:common-range}, we obtain
$$
\left(\sum_i e^{-\beta\lambda_i}\right)\varphi(q)
=\varphi(p)
=\left(\sum_j e^{-\beta\kappa_j}\right)\varphi(q).
$$
It remains only to note that $\varphi(q)\ne0$. Indeed, if $\varphi(q)=0$,
then $\varphi(s_is_i^*)=0$ for all $i$, hence $\varphi(p)=0$. This
contradicts $p+q=1$ and $\varphi(1)=1$. Thus $\varphi(q)>0$, and division
by $\varphi(q)$ yields \eqref{eq:partition}.
\end{proof}

\begin{corollary}\label{cor:beta-one}
For the normalized action \eqref{eq:normalized-action}, if $m\ne n$, the
only possible finite real inverse temperature is $\beta=1$.
\end{corollary}

\begin{proof}
In \eqref{eq:general-action} we have $\lambda_i=\log m$ and
$\kappa_j=\log n$, so that \eqref{eq:partition} becomes
$
me^{-\beta\log m}=ne^{-\beta\log n}$, that is $ (m/n)^{1-\beta}=1 .
$
If $m\ne n$, this is possible only when $1-\beta=0$.
\end{proof}

Thus, in the asymmetric case $m\ne n$, KMS states can occur only at
$\beta=1$. By contrast, in the diagonal case $m=n$, every real value of
$\beta$ satisfies $me^{-\beta\log m}=ne^{-\beta\log n}$; as we shall see
below, $\KMS_\beta$ states indeed exist for every such parameter $\beta$.

The following consequence will fix the choice of measures in the models of
the next section.

\begin{corollary}\label{cor:normalization}
For every $\KMS_1$ state of the normalized dynamics,
$$
\varphi(p)=\varphi(q)=\tfrac12,
\qquad
\varphi(s_is_i^*)=\tfrac1{2m},
\qquad
\varphi(t_jt_j^*)=\tfrac1{2n} .
$$
\end{corollary}

\begin{proof}
At $\beta=1$, Lemma~\ref{lem:basic-kms} gives
$\varphi(s_is_i^*)=\frac1m\varphi(q)$ and
$\varphi(t_jt_j^*)=\frac1n\varphi(q)$. Summing either family gives
$\varphi(p)=\varphi(q)$. Since $p+q=1$, both values are $1/2$, and the
previous formulas for the individual range projections follow.
\end{proof}

The diagonal case $m=n$ and the asymmetric case $m\ne n$ already differ at
the level of traces. We shall exhibit many tracial states on
$\mathcal O_{r,r}$ in Section~\ref{ss:diagonal}, whereas for $m\ne n$ there are
none. This is the familiar obstruction that rules out traces on Cuntz
algebras, and it also drops out of \eqref{eq:partition}.

\begin{corollary}\label{cor:no-traces}
If $m\ne n$, then $\Omn$ has no tracial states.
\end{corollary}

\begin{proof}
For any dynamics the $\KMS_0$ condition is precisely the trace identity,
so a tracial state is a $\KMS_0$ state for \eqref{eq:general-action}.
Setting $\beta=0$ in \eqref{eq:partition} gives $m=n$.
\end{proof}

\section{Full corner reduction}
\label{ss:corner-algebra}

Relations \eqref{rel:common-domain}–\eqref{rel:common-range} show, in particular, that $q$ is equivalent to each of the projections $s_is_i^*$. Thus, $q$ is a full projection, and $\Omn$ can be viewed as a matrix algebra over the corner $q\Omn q$. We show below that restricting to this corner gives a bijection with the corresponding KMS simplex at every inverse temperature, as well as an analogous result for ground states.

Throughout the remainder of this section $m,n\ge1$, $\sigma$ denotes the
normalized dynamics \eqref{eq:normalized-action} on $\Omn$,
and we write
$
  A:=q\Omn q,$ 
$
  p_i:=s_is_i^{*},$
and $
  v_{ij}:=s_i^{*}t_j,$ where the indices $i,j$ are always assumed to range over $1\le i\le m$ and $1\le j\le n.
$
We further set
$$
  u_0:=q,\quad u_i:=s_i,\qquad d_0:=1,\quad d_i:=m,
  \qquad
  Z_\beta:=\sum_{a=0}^{m}d_a^{-\beta}=1+m^{1-\beta},
$$
and finally $\varepsilon:=\log(n/m)$. Note that, if a $\KMS_\beta$ state exists, then by Theorem~\ref{thm:partition-equation}, or equivalently the proof of Corollary~\ref{cor:beta-one}, $Z_\beta=1+m^{1-\beta}=1+n^{1-\beta}.$

\begin{lemma}\label{lem:corner-elementary}
In $\Omn$ the following hold:
\begin{enumerate}[label=\textup{(\roman*)}]
  \item $s_i=ps_iq$ and $t_j=pt_jq$; equivalently
        $s_i^{*}=qs_i^{*}p$ and $t_j^{*}=qt_j^{*}p$.
  \item $qs_i=qt_j=0$ and $s_i^{*}q=t_j^{*}q=0$.
  \item $q,p_1,\dots,p_m$ are pairwise orthogonal projections with
        $q+\sum_{i}p_i=1$, and
        $$
        u_a^{*}u_b=\delta_{ab}q,
        \qquad
        \sum_{a=0}^{m}u_au_a^{*}=1 .
        $$
  \item $s_is_k=s_it_j=t_js_i=t_jt_\ell=0$ for all indices, and likewise
        for the adjoints.
  \item $q\ne0$; in particular $A$ is a unital $C^{*}$-algebra with unit
        $q$.
\end{enumerate}
\end{lemma}

\begin{proof}
\begin{enumerate}[label=\textup{(\roman*)}]
\item From $s_i=s_is_i^{*}s_i=p_is_i$ and $p_i\le p$ we get $ps_i=s_i$,
while $s_iq=s_i(s_i^{*}s_i)=s_i$ by \eqref{rel:common-domain}. The
argument for $t_j$ is the same, using $t_jt_j^{*}\le p$ from
\eqref{rel:common-range}.

\item By (i) and $pq=qp=0$ we have $qs_i=q(ps_i)=(qp)s_i=0$, and
$s_i^{*}q=(qs_i)^{*}=0$.

\item By \eqref{rel:sorth}, $p_ip_k=s_i(s_i^{*}s_k)s_k^{*}=0$ for $i\ne
k$, and $\sum_ip_i=p$ by \eqref{rel:common-range}; now use $pq=0$ and
$p+q=1$. The identities $u_a^{*}u_b=\delta_{ab}q$ are
\eqref{rel:sorth} and \eqref{rel:common-domain} for $a,b\ge1$, and (ii)
together with $q^{2}=q$ otherwise.

\item By (i), $s_is_k=(s_iq)(ps_k)=s_i(qp)s_k=0$, and the remaining cases
are identical.

\item By contradiction, if $q=0$ then $s_i=s_iq=0$ for all $i$, hence $p=\sum_ip_i=0$ and
$1=p+q=0$.
\end{enumerate}
\end{proof}

We next describe the corner $A=q\Omn q$ and the dynamics
induced on it. The elements $v_{ij}$ introduced above provide a convenient family of
generators for $A$, and their relations will play an important role
below. As we will see, the dynamics is now considerably simpler, with the same action on every generator.

\begin{proposition}\label{prop:corner-generators}
The corner $A$ is generated by
the family $\{v_{ij}\}$, which is tame and satisfies
\begin{equation}\label{eq:corner-relations}
  \sum_{j=1}^{n}v_{ij}v_{ij}^{*}=q,
  \qquad
  \sum_{i=1}^{m}v_{ij}^{*}v_{ij}=q.
\end{equation}
Moreover $\sigma$ restricts to a strongly continuous action $\sigma^{A}$
of $\mathbb{R}$ on $A$, and
\begin{equation}\label{eq:corner-dynamics}
  \sigma^{A}_t(v_{ij})=e^{it\varepsilon}v_{ij},
\end{equation}
where $\varepsilon=\log (n/m).$
\end{proposition}

\begin{proof}
The relations \eqref{eq:corner-relations} follow directly from
Lemma~\ref{lem:corner-elementary}(i) and \eqref{rel:common-range}:
$$
  \sum_j v_{ij}v_{ij}^{*}=s_i^{*}\left(\sum_jt_jt_j^{*}\right)s_i
  =s_i^{*}ps_i=s_i^{*}s_i=q,
  \qquad
  \sum_i v_{ij}^{*}v_{ij}=t_j^{*}pt_j=q .
$$
Each $v_{ij}$ is a word in $s_1,\dots,s_m,t_1,\dots,t_n$ and their
adjoints, and every word in the $v_{ij}$ and their adjoints is again such
a word. Tameness of $\{v_{ij}\}$ is therefore inherited from tameness of
$\{s_i,t_j\}$.

For the generation statement, recall that $\Omn$ is the closed
linear span of the words $w=g_1\cdots g_k$ with
$g_r\in\{s_i,t_j,s_i^{*},t_j^{*}\}$, together with $1$. Since $x\mapsto
qxq$ is continuous and $q1q=q$, the algebra $A$ is the closed linear span
of the elements $qwq$. Fix such a word with $qwq\ne0$. By
Lemma~\ref{lem:corner-elementary}(iv) no two consecutive letters of $w$
are both starred or both unstarred, so the letters alternate. If $g_1$
were unstarred then $qg_1=0$ by Lemma~\ref{lem:corner-elementary}(ii), so
$g_1$ is starred and $qg_1=g_1$; if $g_k$ were starred then $g_kq=0$, so
$g_k$ is unstarred and $g_kq=g_k$. Hence $k$ is even, $qwq=w$, and
grouping the letters in consecutive pairs gives
$
  w=(g_1g_2)(g_3g_4)\cdots(g_{k-1}g_k),
$
each pair being of the form $s_i^{*}s_k=\delta_{ik}q$,
$t_j^{*}t_\ell=\delta_{j\ell}q$, $s_i^{*}t_j=v_{ij}$ or
$t_j^{*}s_i=v_{ij}^{*}$. Thus $qwq\in C^{*}(\{v_{ij}\})$, and
$q=\sum_jv_{1j}v_{1j}^{*}$ lies in $C^{*}(\{v_{ij}\})$ as well.

Finally $\sigma_t(q)=q$, so $\sigma_t(A)=A$ and $\sigma^{A}:=\sigma|_A$ is
a strongly continuous action; \eqref{eq:corner-dynamics} is immediate from
\eqref{eq:normalized-action}.
\end{proof}

Proposition~\ref{prop:corner-generators} exhibits generators and relations for the $q$-corner
$A$, but says nothing about further relations among them. The next
proposition settles this by showing a universal property for $A$, showing that the two families in \eqref{eq:corner-relations}, together with tameness, present $A$ completely.
 
\begin{proposition}\label{prop:Auniversal}
Let $B$ be a unital $C^*$-algebra containing a tame family of partial
isometries $\{w_{ij}\}_{i\le m,j\le n}$ with
$
  \sum_{j=1}^n w_{ij}w_{ij}^{*}=1_B$ and $
  \sum_{i=1}^m w_{ij}^{*}w_{ij}=1_B.
$
Then there is a unique unital $*$-homomorphism $\rho\colon A\to B$ with
$\rho(v_{ij})=w_{ij}$. Thus $A$ is the universal unital $C^*$-algebra
generated by a tame family satisfying \eqref{eq:corner-relations}.
\end{proposition}

\begin{proof}
Uniqueness is immediate from Proposition~\ref{prop:corner-generators}, since
the $v_{ij}$ generate $A$. For existence we produce a copy of the defining
relations of $\Omn$ inside $M_{m+1}(B)$.

If a finite sum of projections $\sum_kP_k$ is again a projection $P$, then the
$P_k$ are pairwise orthogonal. By the relations, for fixed $i$ the projections
$\{w_{ij}w_{ij}^{*}\}_j$ are therefore pairwise orthogonal, and for fixed $j$
so are $\{w_{ij}^{*}w_{ij}\}_i$. Consequently, for $i\ne k$,
$$
  w_{ij}w_{kj}^{*}
  =w_{ij}(w_{ij}^{*}w_{ij})(w_{kj}^{*}w_{kj})w_{kj}^{*}=0 ,
$$
and symmetrically $w_{ij}^{*}w_{i\ell}=0$ for $j\ne\ell$.

Write $E_{ab}$, $0\le a,b\le m$, for the matrix units and put
$$
  S_i=E_{i0}\otimes1_B,
  \qquad
  T_j=\sum_{i=1}^{m}E_{i0}\otimes w_{ij},
  \qquad
  Q=E_{00}\otimes1_B,
  \qquad
  P=\sum_{i=1}^{m}E_{ii}\otimes1_B .
$$
Then $S_i^{*}S_k=\delta_{ik}Q$ and $\sum_iS_iS_i^{*}=P$, while the hypothesis and the two
orthogonality relations just derived give
$$
  T_j^{*}T_\ell=E_{00}\otimes\sum_{i}w_{ij}^{*}w_{i\ell}=\delta_{j\ell}Q,
  \qquad
  \sum_{j}T_jT_j^{*}=\sum_{i}E_{ii}\otimes\sum_{j}w_{ij}w_{ij}^{*}=P .
$$
Finally $PQ=0$ and $P+Q=1$, so $\{S_i,T_j\}$ satisfies
\eqref{rel:sorth}--\eqref{rel:pq} with $P,Q$ in place of $p,q$.

For tameness, note that all of $S_i,T_j$ lie in $(E_{10}+\dots+E_{m0})\otimes B$,
so any product of two unstarred letters vanishes, as does any product of two
starred ones, and a nonzero word alternates. Grouping the letters of such a word into consecutive starred--unstarred pairs,
beginning at the first starred letter and leaving at most one letter free at
the end, and using
$$
  S_i^{*}T_j=E_{00}\otimes w_{ij},
  \qquad
  T_j^{*}S_i=E_{00}\otimes w_{ij}^{*},
  \qquad
  S_i^{*}S_k=\delta_{ik}Q,
  \qquad
  T_j^{*}T_\ell=\delta_{j\ell}Q,
$$
every nonzero word takes the form $XVY^{*}$, where $V=E_{00}\otimes u$ with $u$
a word in the $w_{ij}$ and their adjoints or $1_B$, and
$X,Y\in\{Q,S_1,\dots,S_m,T_1,\dots,T_n\}$, the value $Q$ occurring when the
corresponding end carries no free letter. In every case $X^{*}X=Y^{*}Y=Q$. Hence
$$
  \left(XVY^{*}\right)^{*}\left(XVY^{*}\right)
  =YV^{*}\left(X^{*}X\right)VY^{*}
  =Y\left(V^{*}V\right)Y^{*},
$$
which is idempotent because
$$Y(V^{*}V)(Y^{*}Y)(V^{*}V)Y^{*}=Y(V^{*}V)Q(V^{*}V)Y^{*}=Y(V^{*}V)Y^{*},$$ and
self-adjoint. So every word is a partial isometry and $\{S_i,T_j\}$ is tame.

By Proposition~\ref{prop:universal} there is a unital $*$-homomorphism
$\pi\colon\Omn\to M_{m+1}(B)$ with $\pi(s_i)=S_i$ and
$\pi(t_j)=T_j$. Then $\pi(q)=S_1^{*}S_1=Q$ and
$\pi(v_{ij})=S_i^{*}T_j=E_{00}\otimes w_{ij}$, so
$\pi(A)\subseteq QM_{m+1}(B)Q=E_{00}\otimes B$, and writing
$\pi(x)=E_{00}\otimes\rho(x)$ for $x\in A$ gives the required homomorphism $\rho$.
\end{proof}

Note that the relations \eqref{eq:corner-relations} are the diagonal entries of $vv^{*}=I_{m}$ and
$v^{*}v=I_{n}$ for $v=(v_{ij})\in M_{m\times n}(A)$, and the off-diagonal entries are
automatic, since the second paragraph of the proof of Proposition~\ref{prop:Auniversal} shows that a tame family
satisfying \eqref{eq:corner-relations} also satisfies $v_{ij}v_{kj}^{*}=0$ for $i\neq k$
and $v_{ij}^{*}v_{i\ell}=0$ for $j\neq\ell$. Hence $A$ is universal for tame families
making $(v_{ij})$ an $m\times n$ unitary matrix.
Without tameness, the universal algebra for these relations is the known
$U^{\mathrm{nc}}_{m,n}$ \cite{Brown1981, McClanahan1992, McClanahan1993}, whose generators are not partial
isometries for $m,n\geq 2$, so that $A$ is its tame quotient. In particular, the next proposition is the tame counterpart of the isomorphisms
$L_{m,n}\cong M_{m+1}(U^{\mathrm{nc}}_{m,n})\cong M_{n+1}(U^{\mathrm{nc}}_{m,n})$; see e.g. \cite[Introduction]{AEK2013} and the references therein. 

\begin{proposition}\label{prop:matrix-picture}
The map
$$
  \Xi\colon\Omn\to M_{m+1}(A),
  \qquad
  \Xi(x)=\left(u_a^{*}xu_b\right)_{a,b=0}^{m},
$$
is a unital $*$-isomorphism with $\Xi(y)=E_{00}\otimes y$ for $y\in A$ and
$\Xi(p_i)=E_{ii}\otimes q$. Under $\Xi$ the normalized dynamics becomes
\begin{equation}\label{eq:gamma}
  \gamma_t\left((x_{ab})_{ab}\right)
  =\left(\left(d_ad_b^{-1}\right)^{it}\sigma^{A}_t(x_{ab})\right)_{ab},
\end{equation}
where $\gamma_t = \Xi \circ \sigma_t \circ \Xi^{-1}$.
\end{proposition}

\begin{proof}
Each entry $u_a^{*}xu_b$ belongs to $A$, since $u_a^{*}=qu_a^{*}$ and
$u_b=u_bq$ by Lemma~\ref{lem:corner-elementary}. Thus $\Xi$ is well
defined, and it is linear because each entry depends linearly on $x$. Using $1=\sum_cu_cu_c^{*}$, we obtain
$$
  \Xi(xy)_{ab}
  =u_a^{*}x\left(\sum_cu_cu_c^{*}\right)yu_b
  =\sum_c\bigl(u_a^{*}xu_c\bigr)\bigl(u_c^{*}yu_b\bigr)
  =\sum_c\Xi(x)_{ac}\Xi(y)_{cb},
$$
so $\Xi$ is multiplicative. Compatibility with the involution follows
immediately, while $\Xi(1)=(u_a^{*}u_b)_{a,b}=(\delta_{ab}q)_{a,b}$ is the unit of
$M_{m+1}(A)$. To see that $\Xi$ is injective, note first that by Lemma~\ref{lem:corner-elementary}(iii),
$$
        x=\sum_{a,b}u_a\left(u_a^{*}xu_b\right)u_b^{*}
         =\sum_{a,b}u_a\,\Xi(x)_{ab}\,u_b^{*},
$$
so that $x$ is completely determined by $\Xi(x)$. Conversely, given
$(y_{ab})\in M_{m+1}(A)$, the element
$$x=\sum_{a,b}u_ay_{ab}u_b^{*}$$ satisfies $\Xi(x)=(y_{ab})$. Hence $\Xi$
is bijective.

Note that if $y\in A$ and $a\ge1$, then $u_a^{*}y=s_a^{*}qy=0$ and
$yu_a=yqs_a=0$ by Lemma~\ref{lem:corner-elementary}(ii), so every entry of
$\Xi(y)$ outside the $(0,0)$ position vanishes; since
$\Xi(y)_{00}=u_0^{*}yu_0=qyq=y$, this gives $\Xi(y)=E_{00}\otimes y$. Moreover
$u_a^{*}p_iu_b=\delta_{ai}\delta_{bi}q$, since $p_i=u_iu_i^{*}$ and
$u_a^{*}u_i=\delta_{ai}q$ by Lemma~\ref{lem:corner-elementary}(iii); hence
$\Xi(p_i)=E_{ii}\otimes q$.

Finally, for \eqref{eq:gamma}, note that $\sigma_t(u_a)=d_a^{it}u_a$, so
$$
  \sigma^{A}_t\left(\Xi(x)_{ab}\right)
  =\sigma_t(u_a^{*}xu_b)
  =d_a^{-it}d_b^{it}u_a^{*}\sigma_t(x)u_b
  =\left(d_ad_b^{-1}\right)^{-it}\Xi\left(\sigma_t(x)\right)_{ab}.
$$
\end{proof}

Note that interchanging the roles of the two families in the previous proposition yields, in the same way, an isomorphism $\Omn\cong M_{n+1}(A)$ carrying $\sigma$ to the
action \eqref{eq:gamma} with the weights $d_0=1$, $d_j=n$.

The matrix decomposition above allows us to reduce the study of KMS states
on $\Omn$ to the corner $A$. The next theorem gives this
reduction explicitly. In particular, every KMS state on $\Omn$
is completely determined by its restriction to $A$, and every KMS state on
$A$ can be uniquely extended to one on $\Omn$.

For $y$ entire analytic for $\sigma^{A}$, the element $E_{ab}\otimes y$ is
entire analytic for $\gamma$, with
\begin{equation}\label{eq:gamma-analytic}
  \gamma_z\left(E_{ab}\otimes y\right)
  =\left(d_ad_b^{-1}\right)^{iz}E_{ab}\otimes\sigma^{A}_z(y),
  \qquad z\in\C.
\end{equation}
We write $\mathcal D$ for the linear span of such elements. Since the
$\sigma^{A}$-analytic elements form a dense $*$-subalgebra of $A$, $\mathcal D$ is a dense,
$\gamma$-invariant $*$-subalgebra of $M_{m+1}(A)$ of entire
analytic elements. Thus, it suffices to verify the KMS condition on
$\mathcal D$.
\begin{theorem}\label{thm:corner-reduction}
Let $m,n\ge1$ and let $\beta\in\mathbb R$. Then
$
  R_\beta(\varphi):=Z_\beta\varphi|_{A}
$
defines an affine, weak-$*$ homeomorphism
$$
  R_\beta\colon K_\beta\left(\Omn,\sigma\right)
  \xrightarrow{\ \cong\ }
  K_\beta\left(A,\sigma^{A}\right),
$$
whose inverse is given by
\begin{equation}\label{eq:reconstruction}
  \psi\longmapsto\varphi_\psi,
  \qquad
  \varphi_\psi(x)=\frac{1}{Z_\beta}
  \left(\psi(qxq)+m^{-\beta}\sum_{i=1}^{m}\psi\left(s_i^{*}xs_i\right)\right),
  \qquad x\in\Omn.
\end{equation}
\end{theorem}

\begin{proof}
We work in $M_{m+1}(A)$ using Proposition~\ref{prop:matrix-picture} and
identifying $y\in A$ with $E_{00}\otimes y$.

Let $\varphi\in K_\beta(\Omn,\sigma)$. If $y,w\in A$ are
entire analytic for $\sigma^{A}$, then they are also entire analytic for
$\sigma$, with $\sigma_z(y)=\sigma^{A}_z(y)$. The KMS identity therefore
gives
$$
  \varphi(yw)=\varphi(w\sigma^{A}_{i\beta}(y)).
$$
Since such elements are dense in $A$, the positive functional
$\varphi|_A$ satisfies the $\KMS_\beta$ condition for $\sigma^{A}$. By
Lemma~\ref{lem:basic-kms}, applied to $v=s_i$, we have
$\varphi(p_i)=m^{-\beta}\varphi(q)$. Thus
$\varphi(p)=m^{1-\beta}\varphi(q)$, and so
$
  1=\varphi(p)+\varphi(q)=Z_\beta\varphi(q).
$
It follows that $R_\beta(\varphi)(q)=Z_\beta\varphi(q)=1$, so
$R_\beta(\varphi)$ is a state on $A$ and belongs to
$K_\beta(A,\sigma^{A})$. Thus $R_\beta$ is well defined.

We next show that a KMS state is determined by its restriction to $A$.
Let $y\in A$ be entire analytic. For $a\ne b$, take
$X=E_{aa}\otimes q$ and $Y=E_{ab}\otimes y$. Since $X$ is fixed by
$\gamma$, the KMS identity gives
$$
  \varphi\left(E_{ab}\otimes y\right)
  =\varphi(XY)
  =\varphi\left(Y\gamma_{i\beta}(X)\right)
  =\varphi(YX)=0,
$$
because $(E_{ab}\otimes y)(E_{aa}\otimes q)=0$ when $b\ne a$. For the diagonal entries, take $X=E_{a0}\otimes q$ and
$Y=E_{0a}\otimes y$. By \eqref{eq:gamma-analytic},
$\gamma_z(X)=d_a^{iz}X$, so $\gamma_{i\beta}(X)=d_a^{-\beta}X$.
Since $XY=E_{aa}\otimes y$ and $YX=E_{00}\otimes y$, we obtain
$$
  \varphi\left(E_{aa}\otimes y\right)
  =\varphi\left(Y\gamma_{i\beta}(X)\right)
  =d_a^{-\beta}\varphi\left(E_{00}\otimes y\right).
$$
By density, these identities hold for all $y\in A$. Consequently, for
every $x\in\Omn$,
\begin{equation}\tag{$\dagger$}\label{eq:reconstruction-matrix}
  \varphi(x)
  =\sum_{a,b}\varphi\left(E_{ab}\otimes\Xi(x)_{ab}\right)
  =\sum_{a=0}^{m}d_a^{-\beta}
     \varphi\left(E_{00}\otimes u_a^{*}xu_a\right)
  =\varphi(qxq)+m^{-\beta}\sum_{i=1}^{m}
     \varphi\left(s_i^{*}xs_i\right).
\end{equation}
Thus $\varphi$ is completely determined by $\varphi|_A$, so $R_\beta$ is
injective. Moreover, \eqref{eq:reconstruction-matrix} gives exactly
\eqref{eq:reconstruction} when $\psi=R_\beta(\varphi)$.

It remains to show that every $\psi\in K_\beta(A,\sigma^{A})$ arises in
this way. Define a functional on $M_{m+1}(A)$ by
$$
  \varphi\left((x_{ab})_{ab}\right)
  :=\frac{1}{Z_\beta}\sum_{a=0}^{m}d_a^{-\beta}\psi(x_{aa}).
$$
This functional is positive, since
$$
  \varphi(x^{*}x)
  =Z_\beta^{-1}\sum_a d_a^{-\beta}
    \psi\left(\sum_c x_{ca}^{*}x_{ca}\right)\ge0,
$$
and it is unital because
$$
  \varphi(1)=Z_\beta^{-1}\sum_a d_a^{-\beta}\psi(q)=1.
$$
To check the KMS condition, it is enough by linearity and density to
consider $X=E_{ab}\otimes y$ and $Y=E_{cd}\otimes w$ in $\mathcal D$.
Since
$XY=\delta_{bc}E_{ad}\otimes yw$, we have
$$
  \varphi(XY)
  =\delta_{bc}\delta_{ad}
    \frac{d_a^{-\beta}}{Z_\beta}\psi(yw).
$$
On the other hand, \eqref{eq:gamma-analytic} gives
$$
  \gamma_{i\beta}(X)
  =\left(d_ad_b^{-1}\right)^{-\beta}
    E_{ab}\otimes\sigma^{A}_{i\beta}(y),
$$
and hence
$$
  Y\gamma_{i\beta}(X)
  =\left(d_ad_b^{-1}\right)^{-\beta}
    \delta_{da}E_{cb}\otimes
    w\sigma^{A}_{i\beta}(y).
$$
Therefore
$$
  \varphi\left(Y\gamma_{i\beta}(X)\right)
  =\delta_{da}\delta_{cb}
    \left(d_ad_b^{-1}\right)^{-\beta}
    \frac{d_c^{-\beta}}{Z_\beta}
    \psi\left(w\sigma^{A}_{i\beta}(y)\right).
$$
Both expressions vanish unless $b=c$ and $a=d$. In that case, the
scalar factors agree because
$(d_ad_b^{-1})^{-\beta}d_b^{-\beta}=d_a^{-\beta}$, while
$$
  \psi(yw)=\psi(w\sigma^{A}_{i\beta}(y))
$$
follows from the $\KMS_\beta$ condition for $\psi$. Thus
$\varphi\in K_\beta(\Omn,\sigma)$.

Finally,
$\varphi(E_{00}\otimes y)=Z_\beta^{-1}\psi(y)$, so
$R_\beta(\varphi)=\psi$. Translating the definition of $\varphi$ through
$\Xi$ gives \eqref{eq:reconstruction}. Notice also that
$s_i^{*}ys_i=(s_i^{*}q)y(qs_i)=0$ for $y\in A$, so the reconstruction
formula indeed restricts to $Z_\beta^{-1}\psi$ on $A$.

Both $R_\beta$ and the reconstruction map are affine and weak-$*$
continuous, since they are given by evaluation at fixed elements with
coefficients independent of the state. Since they are mutually inverse,
this proves the result.
\end{proof}

\subsection{KMS states in the diagonal case}
\label{ss:diagonal}

Let $m=n=r\ge2$. Then the normalized dynamics is
$\sigma_t(s_i)=r^{it}s_i$ and $\sigma_t(t_j)=r^{it}t_j$, with
$\varepsilon=0$. There are two important differences in this case. First,
the dynamics is inner and has only two spectral levels in the elementary
block decomposition. Second, the dynamics on the corner is trivial, so the
KMS condition in $A$ is just the trace identity for every inverse
temperature. When $m\ne n$, the situation is quite different, and the dynamics is not even approximately inner, as we will see in Corollary~\ref{cor:corner-ground-cases}.

Throughout this subsection we write
$$
h=(\log r)p ,
\qquad
e^{zh}=r^{z}p+q \quad (z\in\C),
$$
an entire family of invertible elements of $\mathcal O_{r,r}$.

\begin{proposition}\label{prop:inner}
If $m=n=r$, the normalized dynamics is implemented by
$u_t=e^{ith}=r^{it}p+q$. In particular every element of
$\mathcal O_{r,r}$ is entire analytic, with $\sigma_z(a)=e^{izh}ae^{-izh}$
for $z\in\C$, and the restriction of $\sigma$ to the two diagonal corners
$p\mathcal O_{r,r}p$ and $q\mathcal O_{r,r}q$ is trivial.
\end{proposition}

\begin{proof}
By Lemma~\ref{lem:corner-elementary}(i), $s_i=ps_iq$ and $t_i=pt_iq$. Thus
$u_ts_iu_t^*=r^{it}s_i$ and $u_tt_iu_t^*=r^{it}t_i$. The generators
therefore transform exactly as under $\sigma_t$, so
$\sigma_t=\Ad(u_t)$. Since $h$ is bounded, $z\mapsto e^{izh}ae^{-izh}$ is
an entire function agreeing with $t\mapsto\sigma_t(a)$ on the real axis.
If $a\in p\mathcal O_{r,r}p$, then $u_tau_t^*=(r^{it}p)a(r^{-it}p)=a$; the
same calculation on the $q$-corner is even simpler, since $u_tq=q$.
\end{proof}

\begin{corollary}\label{cor:corner-diagonal}
Let $r\ge2$. Then $\sigma^{A}$ is trivial, so that
$K_\beta(A,\sigma^{A})=\operatorname{Tr}(A)$ for every $\beta$, and
Theorem~\ref{thm:corner-reduction} gives affine homeomorphisms
$$
  K_\beta\left(\mathcal O_{r,r},\sigma\right)\cong\operatorname{Tr}(A)
  \cong\operatorname{Tr}\left(M_{r+1}(A)\right)
  =\operatorname{Tr}\left(\mathcal O_{r,r}\right),
  \qquad\beta\in\mathbb R .
$$
In particular the $\KMS_\beta$ simplex is, up to affine homeomorphism,
independent of $\beta$, and the problem of determining the KMS states on
$\mathcal O_{r,r}$ reduces to the problem of finding traces.
\end{corollary}

\begin{proof}
Here $\varepsilon=0$, so $\sigma^{A}$ is trivial by
Proposition~\ref{prop:corner-generators}, and for a trivial dynamics the KMS
condition is the trace identity at every $\beta$. Combining this with
Theorem~\ref{thm:corner-reduction} and
Proposition~\ref{prop:matrix-picture} gives, for every $\beta\in\R$,
$$
        K_\beta\left(\mathcal O_{r,r},\sigma\right)
        \xrightarrow[\ \cong\ ]{R_\beta }
        K_\beta\left(A,\sigma^{A}\right)
        =\Tr(A)
        \xrightarrow[\ \cong\ ]{  }
        \Tr\left(M_{r+1}(A)\right)
        \xrightarrow[\ \cong\ ]{\ \widetilde\tau\mapsto\widetilde\tau\circ\Xi \ }
        \Tr\left(\mathcal O_{r,r}\right),
$$
the middle homeomorphism being the canonical identification $\tau_A\mapsto(r+1)^{-1}\sum_a\tau_A(x_{aa})$.
\end{proof}

The composite bijection of Corollary~\ref{cor:corner-diagonal} is the
usual Gibbs correspondence between traces and KMS states, in its
bounded-inner form; compare \cite[Section~5.3]{BR2}. We record it
explicitly, since the normalization matters in what follows.

\begin{corollary}\label{cor:gibbs}
Let $r\ge2$ and fix $\beta\in\mathbb R$. The map
$$
\tau\longmapsto\varphi_\tau,
\qquad
\varphi_\tau(x)=\frac{r+1}{1+r^{1-\beta}}
   \tau\left(e^{-\beta h}x\right),
\qquad x\in\mathcal O_{r,r},
$$
is an affine homeomorphism from $\operatorname{Tr}(\mathcal O_{r,r})$ onto
$K_\beta(\mathcal O_{r,r},\sigma)$.
\end{corollary}

\begin{proof}
For a tracial state $\tau$ we have $\tau(p_i)=\tau(s_i^*s_i)=\tau(q)$ for
each $i$, so that $\tau(p)=r\tau(q)$ and, by $p+q=1$,
$\tau(q)=1/(r+1)$. Hence $\tau\mapsto\psi:=(r+1)\tau|_A$ is the affine
homeomorphism $\operatorname{Tr}(\mathcal O_{r,r})\to\operatorname{Tr}(A)$
of Corollary~\ref{cor:corner-diagonal}. Substituting into
\eqref{eq:reconstruction} and using traciality in the form
$\tau(s_i^{*}xs_i)=\tau(xp_i)$, we obtain
$$
\varphi_\psi(x)
=\frac{r+1}{1+r^{1-\beta}}
 \tau\left(x(q+r^{-\beta}p)\right)
=\frac{r+1}{1+r^{1-\beta}}\tau\left(e^{-\beta h}x\right),
$$
since $e^{-\beta h}=r^{-\beta}p+q$ and $\tau$ is a trace.
\end{proof}

Ara, Exel and Katsura observed that the full algebra $\mathcal O_{r,r}$
admits many finite-dimensional representations; see
\cite[Introduction]{AEK2013}. We now give an explicit family,
parametrized by finite-dimensional unitary tuples, which is particularly
convenient for constructing and distinguishing traces and therefore KMS
states.

\begin{theorem}\label{thm:findim}
Fix $r\ge2$, an integer $d\ge1$, and unitaries $U_1,\dots,U_r\in U(d)$.
Let
$$
K=\C^d,\quad H_p=K^{\oplus r},\quad H_q=K,\quad H=H_p\oplus H_q .
$$
Let $V_i:H_q\to H_p$ be the isometry onto the $i$th copy of $K$, and
define $S_i=V_i$ and $T_i=V_iU_i$. Then $\{S_i,T_i:1\le i\le r\}$ is a
tame family satisfying the relations of $\mathcal O_{r,r}$. Consequently
it defines a representation $\pi_{d,U}:\mathcal O_{r,r}\to B(H)$.
\end{theorem}

\begin{proof}
Let $p_i$ be the projection onto the $i$th copy of $K$ inside $H_p$, put
$p=\sum_ip_i$, and let $q$ be the projection onto $H_q$. Since
$V_i^*V_k=\delta_{ik}1_K$,
$$
S_i^*S_k=\delta_{ik}q,
\qquad
T_i^*T_k=U_i^*V_i^*V_kU_k=\delta_{ik}q,
$$
and $\sum_iS_iS_i^*=p=\sum_iV_iU_iU_i^*V_i^*=\sum_iT_iT_i^*$, with
$p+q=1_H$ and $pq=0$. Thus the defining relations hold.

For tameness, regard each $V_i$ as an operator on $H$, vanishing on $H_p$,
identify operators on $H_q=K$ with their extensions by zero, and let
$\Gamma=\langle U_1,\dots,U_r\rangle\subseteq U(K)$. Put
$$
\mathcal W=\{0\}\cup\Gamma\cup V_i\Gamma\cup\Gamma V_j^*\cup V_i\Gamma V_j^*
\qquad(1\le i,j\le r),
$$
which contains $S_i=V_i$, $T_i=V_iU_i$ and their adjoints. Each element of
$\Gamma$ is a unitary supported on $H_q$ and $V_i^*V_j=\delta_{ij}q$, so
any product of two elements of $\mathcal W$ either vanishes, for lack of a
matching source and range, or absorbs the middle
$V_j^*V_k=\delta_{jk}q$ and is again of one of the four displayed forms;
closure under adjoints is immediate. Finally each nonzero element of
$\mathcal W$ is a partial isometry, with
$$
(V_igV_j^*)^*(V_igV_j^*)=p_j,
\qquad
(V_igV_j^*)(V_igV_j^*)^*=p_i,
$$
and the analogous identities for $g$, $V_ig$ and $gV_j^*$, obtained by
reading $V_i$ or $V_j^*$ as absent. Hence every word in the generators and
their adjoints lies in $\mathcal W$ and is a partial isometry, so
$\{S_i,T_i\}$ is tame and the universal property of $\mathcal O_{r,r}$
yields $\pi_{d,U}$.
\end{proof}

\begin{remark}\label{rem:findim-vs-reduced}
Note that the representations in Theorem~\ref{thm:findim} concern the full
tame algebra $\mathcal O_{r,r}$. They do not contradict
\cite[Theorem~9.5]{AEK2013}, which proves that the reduced algebra
$\Omn^{r}$ admits no nonzero finite-dimensional
representations when $m,n\ge2$.
\end{remark}

\begin{corollary}\label{cor:findimtraces}
Let $r\ge 2$. With the notation of Theorem~\ref{thm:findim}, the state
$$
\tau_{d,U}=\frac{1}{(r+1)d}\mathrm{Tr}\circ\pi_{d,U}
$$
is a tracial state on $\mathcal O_{r,r}$, and for every
$\beta\in\mathbb R$ the corresponding $\KMS_\beta$ state of
Corollary~\ref{cor:gibbs} is the Gibbs state
$$
\varphi_{\beta,d,U}(x)
 =\frac{\mathrm{Tr}(e^{-\beta H_{d}}\pi_{d,U}(x))}
        {\mathrm{Tr}(e^{-\beta H_{d}})},
\qquad
H_d=(\log r)\pi_{d,U}(p),
$$
with $\mathrm{Tr}(e^{-\beta H_d})=d(1+r^{1-\beta})$.
\end{corollary}

\begin{proof}
$\mathrm{Tr}$ is a trace on the finite-dimensional algebra $B(H)$ and
$\pi_{d,U}$ is a $*$-homomorphism, so $\tau_{d,U}$ is tracial. Moreover
$\tau_{d,U}(1)=\mathrm{Tr}(1_H)/((r+1)d)=1$. Since
$\pi_{d,U}(e^{-\beta h})=e^{-\beta H_d}$ and
$\mathrm{Tr}(e^{-\beta H_d})=d+r^{-\beta}rd=d(1+r^{1-\beta})$,
Corollary~\ref{cor:gibbs} gives
$$
\varphi_{\tau_{d,U}}(x)
=\frac{r+1}{1+r^{1-\beta}}\cdot
 \frac{\mathrm{Tr}(e^{-\beta H_d}\pi_{d,U}(x))}{(r+1)d}
=\frac{\mathrm{Tr}(e^{-\beta H_d}\pi_{d,U}(x))}
       {\mathrm{Tr}(e^{-\beta H_d})} . 
$$
\end{proof}

\begin{corollary}\label{cor:nonunique}
For $r\ge2$ and every $\beta\in\mathbb R$, the $\KMS_\beta$ states of
$(\mathcal O_{r,r},\sigma)$ exist and are not unique.
\end{corollary}

\begin{proof}
Existence is Corollary~\ref{cor:findimtraces}. For non-uniqueness, take $d=1$,
so that each $U_i$ is a scalar in $\T$. In this representation
$\pi_{1,U}(s_i^*t_i)=S_i^*T_i=U_i$ on the $q$-summand and vanishes on the
$p$-summand, hence
$$
\varphi_{\beta,1,U}(s_i^*t_i)=\frac{U_i}{1+r^{1-\beta}} .
$$
Choosing $U_1=1$ and $U_1=-1$, with the remaining $U_i$ equal to $1$ in both
cases, gives two states whose values at $s_1^*t_1$ are
$\pm(1+r^{1-\beta})^{-1}$, hence two distinct $\KMS_\beta$ states at every
$\beta\in\R$.
\end{proof}

\begin{example}\label{ex:O22}
Take $r=2$ and $d=1$ in Theorem~\ref{thm:findim}. Let $U_1=1$ and
$U_2=e^{i\theta}$. The Hilbert space is $\C^2\oplus\C$, the projection $p$
has rank $2$, and $q$ has rank $1$. The Gibbs denominator is
$$
        \Tr(e^{-\beta(\log2)p})=\Tr(q+2^{-\beta}p)=1+2^{1-\beta} ,
$$
so that, for every real $\beta$,
$$
        \varphi_{\beta,\theta}(s_2^*t_2)
        =\frac{e^{i\theta}}{1+2^{1-\beta}} .
$$
Thus $\theta=0$ and $\theta=\pi$ give different $\KMS_\beta$ states at the
same inverse temperature.
\end{example}

The traces produced above are far from exhausting $\Tr(\mathcal O_{r,r})$.
We show below that $A$ has many quotients isomorphic to $C^*(\mathbb F_r)$, so that the trace simplex of the latter embeds affinely in
$\Tr(\mathcal O_{r,r})$. Under this identification, the traces produced above are exactly those coming from the finite-dimensional unitary representations of $\mathbb F_r$. We remark that, for the identity
permutation, a related statement appears in \cite[Example~6.6]{AraLolk2018}.

\begin{proposition}\label{prop:free-group-traces}
Let $r\ge2$, let $\Sigma_r$ be the permutation group on $r$ elements, and let $\varpi \in \Sigma_r$. The closed ideal $J_\varpi$ of $A$ generated by
$\{v_{ij}:j\ne\varpi(i)\}$ is proper, and $A/J_\varpi\cong C^*(\mathbb F_r)$, the
image of $v_{i\varpi(i)}$ corresponding to the $i$th canonical unitary generator.
Consequently $\Tr\left(C^*(\mathbb F_r)\right)$ embeds affinely in
$\Tr(\mathcal O_{r,r})$ as a face, and $\Tr(\mathcal O_{r,r})$ is
infinite-dimensional. Under this identification, the traces of Corollary~\ref{cor:findimtraces}
correspond to the finite-dimensional unitary representations of $\mathbb F_r$;
the remaining traces of $C^*(\mathbb F_r)$, such as the regular one, give
further $\KMS_\beta$ states of $\mathcal O_{r,r}$.
\end{proposition}

\begin{proof}
    We first treat $\varpi=\mathrm{id}$ and write $J=J_{\mathrm{id}}$. Let
$g_1,\dots,g_r$ be the canonical unitary generators of the full group
$C^*$-algebra $C^*(\mathbb F_r)$ and put $w_{ij}=\delta_{ij}g_i$. This family
is tame, every word in the $w_{ij}$ and their adjoints being $0$ or a word in
the $g_i^{\pm1}$, and it satisfies the hypotheses of
Proposition~\ref{prop:Auniversal}, both sums having a single nonzero term.
There is therefore a unital $*$-homomorphism $\rho\colon A\to C^*(\mathbb F_r)$
with $\rho(v_{ij})=\delta_{ij}g_i$; it is surjective, since its range contains
the $g_i$, and it kills $v_{ij}$ for $i\ne j$. Hence $J\subseteq\ker\rho$, and
$\rho$ induces a surjection $\bar\rho\colon A/J\to C^*(\mathbb F_r)$.

Write $\bar v_i$ for the image of $v_{ii}$ in $A/J$. There the relations
\eqref{eq:corner-relations} have a single nonzero term on each side, so
$\bar v_i\bar v_i^{*}=\bar v_i^{*}\bar v_i=q$ and the $\bar v_i$ are unitaries;
they generate $A/J$ by Proposition~\ref{prop:corner-generators}. The universal
property of $C^*(\mathbb F_r)$ therefore gives a unital $*$-homomorphism
$\theta\colon C^*(\mathbb F_r)\to A/J$ with $\theta(g_i)=\bar v_i$. Since
$\bar\rho\circ\theta$ and $\theta\circ\bar\rho$ fix the generators of their
domains, both are the identity, and $A/J\cong C^*(\mathbb F_r)$; in particular
$A/J\ne0$, so $J$ is proper.

For general $\varpi\in\Sigma_r$, permuting $t_1,\dots,t_r$ defines by
Proposition~\ref{prop:universal} an automorphism $\Theta_\varpi$ of
$\mathcal O_{r,r}$ with $\Theta_\varpi(s_i)=s_i$ and
$\Theta_\varpi(t_j)=t_{\varpi(j)}$; it is equivariant because
$\sigma_t(t_j)=r^{it}t_j$ for every $j$. It fixes $q$, hence restricts to an
automorphism of $A$ with $\Theta_\varpi(v_{ij})=v_{i\varpi(j)}$, which carries
$J_{\mathrm{id}}$ onto $J_\varpi$. Composing with $\bar\rho$ we obtain an
isomorphism $\bar\rho_\varpi\colon A/J_\varpi\to C^*(\mathbb F_r)$ carrying the
image of $v_{i\varpi(i)}$ to $g_i$.

Fix $\varpi\in\Sigma_r$ and let $\varsigma$ be a trace of $C^*(\mathbb F_r)$.
Then $\psi_\varsigma:=\varsigma\circ\bar\rho_\varpi\circ\pi_{J_\varpi}$ is a
tracial state of $A$, and distinct $\varsigma$ give distinct $\psi_\varsigma$
because $\pi_{J_\varpi}$ is surjective; the assignment is clearly affine. By
Corollary~\ref{cor:corner-diagonal} it therefore embeds
$\Tr(C^*(\mathbb F_r))$ affinely into $\Tr(\mathcal O_{r,r})$. The image is the
set of traces vanishing on $J_\varpi$, which is a face (if a convex combination
of traces vanishes on $J_\varpi$, so does each summand, by positivity).

For $\lambda\in\T^{r}$ let $\chi_\lambda$ be the character of
$C^*(\mathbb F_r)$ with $\chi_\lambda(g_i)=\lambda_i$, obtained from the
universal property applied to the unitaries $\lambda_i\in\T$. Being
multiplicative, each $\chi_\lambda$ is an extreme trace, and distinct $\lambda$
give distinct characters. Extreme points of a face being extreme points of the
whole simplex, $\Tr(\mathcal O_{r,r})$ has infinitely many extreme points and
is therefore infinite-dimensional.

Finally, with the notation of Theorem~\ref{thm:findim} we have
$\pi_{d,U}(v_{ij})=S_i^{*}T_j=\delta_{ij}U_i$, since $V_i^{*}V_j=\delta_{ij}1_K$.
Hence $\pi_{d,U}$ annihilates $J_{\mathrm{id}}$, and the trace $\tau_{d,U}$ of
Corollary~\ref{cor:findimtraces} lies in the face indexed by
$\varpi=\mathrm{id}$; the corresponding tracial state of $A$ is
$d^{-1}\Tr\circ\pi_{d,U}|_A$, which under $\bar\rho_{\mathrm{id}}$ becomes the
normalized character $d^{-1}\Tr\circ\Lambda_U$ of the representation
$\Lambda_U\colon\mathbb F_r\to U(d)$ with $\Lambda_U(g_i)=U_i$. By the universal property of free groups, every
finite-dimensional unitary representation of $\mathbb F_r$ arises this way, so
the traces of Corollary~\ref{cor:findimtraces} are exactly the normalized
characters of such representations. They are far from exhausting
$\Tr(C^*(\mathbb F_r))$: the regular trace, for one, is not among them. Composing with the
automorphisms $\Theta_\varpi$ moves all of this to the remaining faces.
\end{proof}

\subsection{KMS states in the boundary cases}\label{ss:boundary}

We now treat the cases in which $m$ or $n$ equals $1$. It is known that
$\mathcal O_{1,n}$ and $\mathcal O_{n,1}$ are both isomorphic to
$M_2(\mathcal O_n)$ \cite[p.~1292]{AEK2013}; in particular, they are
Morita--Rieffel equivalent to the Cuntz algebra $\mathcal O_n$. In the
corner both the isomorphism and the identification of the
dynamics under it are immediate, because for $m=1$ the corner relations
\eqref{eq:corner-relations} degenerate into the Cuntz relations and the
weights $d_0=d_1=1$ make the twist in \eqref{eq:gamma} disappear.

\begin{corollary}\label{cor:corner-boundary}
Let $n\ge2$. Then $A\cong\mathcal O_n$, the isomorphism carrying $v_{1j}$
to the $j$th canonical generator, and
Proposition~\ref{prop:matrix-picture} gives
$\mathcal O_{1,n}\cong M_2(\mathcal O_n)$ with $\sigma$ corresponding to
$\id_{M_2}\otimes\gamma$, where $\gamma_t(w_j)=n^{it}w_j$ is the usual
Cuntz gauge dynamics. Consequently $\mathcal O_{1,n}$ has a unique
$\KMS_1$ state and no $\KMS_\beta$ state for $\beta\ne1$. The case
$\mathcal O_{n,1}$ is analogous, interchanging the two families.
\end{corollary}

\begin{proof}
Write $v_j:=v_{1j}$. The second relation in \eqref{eq:corner-relations}
has a single summand, so $v_j^{*}v_j=q$, while the first reads
$\sum_jv_jv_j^{*}=q$. If a finite sum of projections is a projection,
then the individual projections are pairwise orthogonal; hence the ranges
of the $v_j$ are pairwise orthogonal and $\{v_j\}$ is a Cuntz family with
unit $q$. Since these elements generate $A$
(Proposition~\ref{prop:corner-generators}) and $A\ne0$
(Lemma~\ref{lem:corner-elementary}(v)), simplicity of $\mathcal O_n$ gives
$A\cong\mathcal O_n$. For $m=1$ we have $d_0=d_1=1$, so the twist in
\eqref{eq:gamma} is trivial and $\gamma=\id_{M_2}\otimes\sigma^{A}$, with
$\sigma^{A}_t(v_j)=n^{it}v_j$ by \eqref{eq:corner-dynamics}. The last
assertion now follows from Theorem~\ref{thm:corner-reduction} together
with the uniqueness theorem for the Cuntz gauge action
\cite{OP1978,EFW1984}.
\end{proof}

\begin{example}\label{ex:boundary-O12}
In $\mathcal O_{1,2}$, put $v_j=s^*t_j$. Then $v_1,v_2$ are Cuntz
generators in the $q$-corner and
$
        \sigma_t(v_j)=s^*2^{it}t_j=2^{it}v_j .
$
Thus $\mathcal O_{1,2}\cong M_2(\mathcal O_2)$ with dynamics
$\id_{M_2}\otimes\gamma$, and the unique KMS state is
$\operatorname{tr}_2\otimes\varphi_2$, where $\varphi_2$ is the usual
Cuntz KMS state.
\end{example}

Finally we consider the exceptional point $(m,n)=(1,1)$. Here both
energies $\log m$ and $\log n$ vanish, so the normalized dynamics is
trivial on all of $\mathcal O_{1,1}$ and the KMS condition reduces to the
trace condition at every inverse temperature. The algebra itself is also
elementary, since the corner is generated by the single unitary $v_{11}=s^*t$.

\begin{proposition}\label{prop:O11}
There is an isomorphism $\mathcal O_{1,1}\cong M_2(C(\T))$ under which the
normalized dynamics is trivial. Hence, for every real $\beta$, the
$\KMS_\beta$ states are exactly the tracial states, equivalently the
states obtained by integrating the normalized matrix trace against a
Borel probability measure on $\T$.
\end{proposition}

\begin{proof}
Write $s,t$ for the two generators and $v=v_{11}=s^{*}t$. By
\eqref{eq:corner-relations}, $vv^{*}=v^{*}v=q$, so $v$ is a unitary of $A$, and
$A=C^{*}(v)$ by Proposition~\ref{prop:corner-generators}. By
Proposition~\ref{prop:Auniversal}, $A$ is universal for a single unitary,
hence $A\cong C(\T)$ with $v\mapsto z$, and
$\mathcal O_{1,1}\cong M_2(C(\T))$ by
Proposition~\ref{prop:matrix-picture}. Since $\log1=0$ the dynamics is trivial,
and for a trivial dynamics the KMS condition is precisely the trace identity.
\end{proof}

\subsection{Ground states}
\label{ss:ground}

Recall that a state $\varphi$ on a $C^*$-dynamical system $(B,\sigma)$ is
a \emph{ground state} if, for all entire analytic elements $a,b\in B$, the
function $z\mapsto\varphi\left(a\sigma_z(b)\right)$ is bounded on the
upper half-plane $\{z\in\C:\operatorname{Im}z\ge0\}$. For almost-periodic
$C^*$-dynamical systems, ground states can be described more generally as
states on an associated crystal $C^*$-algebra, as shown in
\cite{LacaNeshveyevYamashita2025};
Corollary~\ref{cor:corner-ground-cases} below is consistent with that
framework, since it identifies the ground-state simplex directly with the
state space of the corner.

The following criterion is a specialization of the ground-state criterion
for graded $C^*$-algebras in \cite[Proposition~1.5(ii)]{ExelLaca2003}. We
include the short proof in the present notation.

\begin{lemma}\label{lem:ground-spectral}
Let $(B,\sigma)$ be a $C^*$-dynamical system and let $B_0$ be a
$\sigma$-invariant dense $*$-subalgebra linearly spanned by homogeneous
entire analytic elements. If $\varphi$ is a ground state, then
$\varphi(ab)=0$ whenever $a\in B_0$ is homogeneous and $b\in B_0$ is
homogeneous of negative energy, meaning that $\sigma_t(b)=e^{itE}b$ with
$E<0$.
\end{lemma}

\begin{proof}
For a fixed homogeneous element $a$ we have
$\varphi(a\sigma_z(b))=e^{izE}\varphi(ab)$. If $z=t+iy$ with $y\ge0$, then
$|e^{izE}|=e^{-yE}=e^{y|E|}$, which is unbounded as $y\to\infty$.
Boundedness is therefore possible only if $\varphi(ab)=0$.
\end{proof}

We shall apply the lemma inside the corner. Let $A_0$ denote the
$*$-algebra generated algebraically by the $v_{ij}$; it is dense in $A$ by
Proposition~\ref{prop:corner-generators}, invariant under $\sigma^{A}$,
and spanned by words in the $v_{ij}$ and their adjoints, each of which is
homogeneous of energy an integer multiple of $\varepsilon$ by
\eqref{eq:corner-dynamics}. Thus $A_0$ satisfies the hypotheses of
Lemma~\ref{lem:ground-spectral}.

\begin{proposition}\label{prop:corner-ground}
Assume $(m,n)\neq(1,1)$. Then every ground state $\varphi$ of
$(\Omn,\sigma)$ satisfies $\varphi(p)=0$, and
$\varphi\mapsto\varphi|_A$ is an affine weak-$*$ homeomorphism from the
ground-state space of $(\Omn,\sigma)$ onto that of $(A,\sigma^{A})$,
with inverse $\psi\mapsto\psi(q\,\cdot\,q)$.
\end{proposition}

\begin{proof}
Interchanging the two families if necessary, we may assume $m\ge2$, so
that $d_i=m>1$ for all $i\ge 1$.

Let $\varphi$ be a ground state. Fix $i\ge1$ and put $X=E_{i0}\otimes q$,
$Y=E_{0i}\otimes q$. By \eqref{eq:gamma-analytic},
$\gamma_z(Y)=m^{-iz}Y$, so that $|m^{-iz}|=m^{y}$ for $z=x+iy$, which is
unbounded as $y\to\infty$; since $XY=E_{ii}\otimes q=\Xi(p_i)$,
boundedness of $z\mapsto\varphi\left(X\gamma_z(Y)\right)$ forces
$\varphi(p_i)=0$. Summing over $i$ gives $\varphi(p)=0$. Since $p$ is a
projection, Cauchy--Schwarz gives
$$
        \left|\varphi(pa)\right|^{2}
        \le\varphi\left(p^{*}p\right)\varphi\left(a^{*}a\right)
        =\varphi(p)\,\varphi\left(a^{*}a\right)=0
        \qquad (a\in\Omn),
$$
so $\varphi$ vanishes on $p\Omn$, and on $\Omn p$ by taking adjoints.
Expanding $x=(p+q)x(p+q)$, the three terms involving $p$ therefore vanish and
$\varphi(x)=\varphi(qxq)$, so $\varphi|_A$ is a state of $A$. It is a ground
state for $\sigma^{A}$, because $\sigma^{A}$-analytic elements of $A$ are
$\sigma$-analytic with the same orbit.

Conversely, let $\psi$ be a ground state of $(A,\sigma^{A})$ and put
$\varphi(x):=\psi(qx q)$, a state of $\Omn$. Let
$a,b\in\Omn$ be entire analytic and set $x_c:=qau_c$ and
$y_c:=u_c^{*}bq$, which lie in $A$ and are entire analytic for
$\sigma^{A}$, since $\sigma_z(x_c)=d_c^{iz}q\sigma_z(a)u_c\in A$ and
$\sigma_z(y_c)=d_c^{-iz}u_c^{*}\sigma_z(b)q\in A$. Inserting
$1=\sum_cu_cu_c^{*}$,
$$
  \varphi\left(a\sigma_z(b)\right)
  =\psi\left(qa\sigma_z(b)q\right)
  =\sum_{c=0}^{m}d_c^{iz}\psi\left(x_c\sigma^{A}_z(y_c)\right).
$$
For $z=x+iy$ with $y\ge0$ we have $|d_c^{iz}|=d_c^{-y}\le1$, because
$d_c\ge1$, while each function
$z\mapsto\psi(x_c\sigma^{A}_z(y_c))$ is bounded on the upper
half-plane since $\psi$ is a ground state. Hence
$z\mapsto\varphi\left(a\sigma_z(b)\right)$ is bounded and $\varphi$ is a
ground state. The two maps are mutually inverse, affine and weak-$*$
continuous.
\end{proof}

\begin{corollary}\label{cor:corner-ground-cases}
For the normalized dynamics on $\Omn$ the ground-state space
is as follows.
\begin{enumerate}[label=\textup{(\alph*)}]
\item If $m\ne n$, there are no ground states; hence $\sigma$ is not
      approximately inner, by \cite[Theorem~2.3]{PowersSakai}.
\item If $m=n=r\ge2$, the ground states are precisely the states
      satisfying $\varphi(p)=0$, and restriction to $q\mathcal O_{r,r}q$
      is an affine homeomorphism onto the state space of that corner.
\item If $(m,n)=(1,1)$, the normalized dynamics is trivial and every
      state is a ground state.
\end{enumerate}
\end{corollary}

\begin{proof}
(a) Suppose first $m>n$, so that $\varepsilon<0$. If
$(\Omn,\sigma)$ had a ground state, then by
Proposition~\ref{prop:corner-ground} so would $(A,\sigma^{A})$, say
$\psi$. Applying Lemma~\ref{lem:ground-spectral} with $a=v_{ij}^{*}$ and
$b=v_{ij}$ gives $\psi(v_{ij}^{*}v_{ij})=0$; summing over $i$ and
using the second relation in \eqref{eq:corner-relations} yields
$\psi(q)=0$, a contradiction. If $n>m$ then $\varepsilon>0$ and one argues
in the same way with $a=v_{ij}$, $b=v_{ij}^{*}$, summing over $j$ and
using the first relation instead.

(b) Here $\varepsilon=0$, so $\sigma^{A}$ is trivial and every state of
$A$ is a ground state; now apply Proposition~\ref{prop:corner-ground},
whose proof also shows that the ground states are exactly the states
vanishing on $p$.

(c) By Proposition~\ref{prop:O11} the dynamics is trivial, so
$z\mapsto\varphi(a\sigma_z(b))=\varphi(ab)$ is constant.
\end{proof}

\begin{example}\label{ex:ground-O22}
For $\mathcal O_{2,2}$ the Hamiltonian implementing the normalized
dynamics is $(\log2)p$, so that $q\mathcal O_{2,2}q$ is the lower spectral
corner and a state is ground exactly when it vanishes on $p$. In the
finite-dimensional representation of Theorem~\ref{thm:findim} with $d=1$,
the vector state on the $H_q$ summand gives $\varphi(p)=0$,
$\varphi(q)=1$ and $\varphi(s_i^*t_i)=U_i$. Changing the scalar
$U_i\in\T$ changes the ground state.
\end{example}

\section{The transportation polytope of $\KMS_1$ states}
\label{sec:asymmetric}

Throughout this section, we assume that $m,n\geq 2$, unless stated otherwise. By
Corollary~\ref{cor:beta-one}, KMS states in the asymmetric case $m\neq n$ can only occur at $\beta=1$, and
there $Z_1=1+m^{0}=2$, so Theorem~\ref{thm:corner-reduction} specializes to an
affine weak-$*$ homeomorphism
\begin{equation}\label{eq:R1}
        R_1\colon K_1(\Omn,\sigma)\xrightarrow{\ \cong\ }
        K_1(A,\sigma^A),
        \qquad
        R_1(\varphi)=2\varphi|_A .
\end{equation}
Every question about $\KMS_1$ states of $\Omn$ is therefore a
question about the system $(A,\sigma^A)$, in which the dynamics has one
energy $\varepsilon=\log(n/m)$. 

Existence of $\KMS_1$ in this regime is established in Section~\ref{sec:regular-models} by
constructing models over the free group on $m+n$ generators. The same construction applies to $\KMS_1$ states in the diagonal case; no relation between $m$ and $n$ will be assumed. 
In Section~\ref{sec:nonunique} we show that
the $\KMS_1$ simplex surjects onto a 2-way transportation polytope whose
defining constraints are given by \eqref{eq:corner-relations}, and in
Section~\ref{sec:fibres} we identify the fibres over the vertices with $\KMS_1$
simplices of quotients of $A$ by ideals generated by generators.

\subsection{Existence through conformal branching models}
\label{sec:regular-models}

Let $X=P\sqcup Q$ be a standard non-atomic probability space with
$\mu(P)=\mu(Q)=\frac12$, and choose measurable partitions
$$
        P=P_1\sqcup\cdots\sqcup P_m=R_1\sqcup\cdots\sqcup R_n
$$
with $\mu(P_i)=\frac1{2m}$ and $\mu(R_j)=\frac1{2n}$; these are precisely the
values that Corollary~\ref{cor:normalization} forces on any $\KMS_1$
state. Next, choose measurable isomorphisms
$$
        \theta_i:Q\to P_i,
        \qquad
        \eta_j:Q\to R_j
$$
such that, for every measurable $E\subseteq Q$,
\begin{equation}\label{eq:conformal-basic}
        \mu(\theta_i(E))=\frac1m\mu(E),
        \qquad
        \mu(\eta_j(E))=\frac1n\mu(E).
\end{equation}

These maps exist by the isomorphism theorem for non-atomic standard probability
spaces: after normalizing the restricted measures, $Q$, $P_i$ and $R_j$ are all
isomorphic modulo null sets to $([0,1],\lambda)$; see, for example,
\cite[Sect.~17.F]{Kechris1995}. Undoing the normalization gives isomorphisms
satisfying \eqref{eq:conformal-basic}. These are defined only modulo null sets,
so we fix bimeasurable representatives on conull Borel subsets, simultaneously
for the finitely many maps involved.

We call such a system $(X,\mu,\{P_i\},\{R_j\},\{\theta_i\},\{\eta_j\})$ a
\emph{conformal branching model}. See Figure~\ref{fig:branching_model} for a
schematic representation.

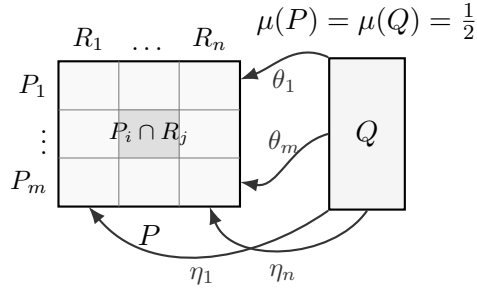
\begin{figure}[H]
\centering
\begin{tikzpicture}[
scale=0.6,
    >=Latex,
    box/.style={draw, thick, minimum width=1cm, minimum height=2cm},
    gridline/.style={draw=black!50, thin}
]
    \coordinate (P_sw) at (0, -1.6);
    \coordinate (P_ne) at (4,  1.6);
    \draw[thick, fill=gray!5] (P_sw) rectangle (P_ne);
    \fill[gray!25] (1.33, -0.53) rectangle (2.66, 0.53);
    \node[font=\scriptsize] at (2, 0) {$P_i \cap R_j$};
    \draw[gridline] (0,  0.53) -- (4,  0.53);
    \draw[gridline] (0, -0.53) -- (4, -0.53);
    \node[left, font=\footnotesize] at (0,  1.06) {$P_1$};
    \node[left, font=\footnotesize] at (0,  0)    {$\vdots$};
    \node[left, font=\footnotesize] at (0, -1.06) {$P_m$};
    \draw[gridline] (1.33, -1.6) -- (1.33, 1.6);
    \draw[gridline] (2.66, -1.6) -- (2.66, 1.6);
    \node[above, font=\footnotesize] at (0.66, 1.6) {$R_1$};
    \node[above, font=\footnotesize] at (2.0,  1.6) {$\dots$};
    \node[above, font=\footnotesize] at (3.33, 1.6) {$R_n$};
    \node[below=0.1cm, font=\small] at (2, -1.6) {$P$};
    \node[box, fill=gray!8] (Q) at (6.8, 0) {};
    \node[above=0.1cm of Q, font=\small] {$\mu(P)=\mu(Q)=\tfrac{1}{2}$};
    \node[font=\small] at (6.8, 0) {$Q$};
    \draw[->, thick, black!80]
        (Q.north west) to[out=155, in=20]
        node[midway, below, font=\footnotesize] {$\theta_1$} (4, 1.2);
    \draw[->, thick, black!80]
        (Q.west) to[out=195, in=-10]
        node[midway, above, font=\footnotesize] {$\theta_m$} (4, -1.06);
    \draw[->, thick, black!80]
        (Q.south west) to[out=215, in=-50]
        node[midway, below, font=\footnotesize] {$\eta_1$} (0.66, -1.6);
    \draw[->, thick, black!80]
        (Q.south) to[out=235, in=-70]
        node[midway, below=0.05cm, font=\footnotesize] {$\eta_n$} (3.33, -1.6);
\end{tikzpicture}
\caption{A conformal branching model for $X=P\sqcup Q$. The grid on $P$ displays
the two partitions $P=P_1\sqcup\cdots\sqcup P_m$ and $P=R_1\sqcup\cdots\sqcup
R_n$ side by side. At this stage of the construction only
$\mu(P_i)=\frac1{2m}$ and $\mu(R_j)=\frac1{2n}$ are fixed. The overlaps
$\mu(P_i\cap R_j)$, one of which is shaded, are not yet determined; they become
the data of Section~\ref{sec:nonunique}, and their pullbacks to $Q$ are the
sets $Q_{ij}$ of Proposition~\ref{prop:corner-model}.}
\label{fig:branching_model}
\end{figure}

Let $G=\mathbb F(a_1,\ldots,a_m,b_1,\ldots,b_n)$ be the free group on $m+n$
generators, one for each of the maps $\theta_i,\eta_j$, and write
$\mathcal L=\{a_1^{\pm1},\ldots,a_m^{\pm1},b_1^{\pm1},\ldots,b_n^{\pm1}\}$ for
the set of letters. To each letter we associate a bimeasurable partial
bijection,
$$
        \vartheta_{a_i}=\theta_i,
        \qquad
        \vartheta_{b_j}=\eta_j,
        \qquad
        \vartheta_{\ell^{-1}}=\vartheta_\ell^{-1},
$$
and to each $g\in G$ with reduced expression $g=\ell_1\cdots\ell_k$ the
composition $\alpha_g=\vartheta_{\ell_1}\circ\cdots\circ\vartheta_{\ell_k}$,
defined on its maximal natural domain; we also set $\alpha_e=\id_X$. We retain
the label $g$ even when two of these partial maps happen to agree on a
measurable subset.

Each $\alpha_g$ has thus been defined separately, one reduced word at a time;
the next lemma checks that they fit together into a partial action of $G$ on
$X$.

\begin{lemma}\label{lem:measurable-partial-action}
For $g\in G$, put $X_g=\operatorname{ran}(\alpha_g)$. Then each $X_g$ is
measurable and $\alpha_g:X_{g^{-1}}\to X_g$ is a bimeasurable nonsingular
bijection. Moreover,
\begin{align}
        \alpha_{g^{-1}}&=\alpha_g^{-1},
        \label{eq:partial-inverse}\\
        \alpha_g(X_{g^{-1}}\cap X_h)&=X_g\cap X_{gh},
        \label{eq:partial-domains}\\
        \alpha_g(\alpha_h(x))&=\alpha_{gh}(x)\label{eq:partial-composition}
\end{align}
for every $x\in X_{h^{-1}}\cap X_{h^{-1}g^{-1}}$. Thus
$\{X_g,\alpha_g\}_{g\in G}$ is a measurable partial action of $G$ on $X$.
\end{lemma}

\begin{proof}
A composition of bimeasurable partial bijections is again a bimeasurable
partial bijection on a measurable domain. If $\rho\colon A\to B$ and
$\kappa\colon C\to D$ are such maps, the natural domain of $\rho\circ\kappa$ is
$C\cap\kappa^{-1}(A)$, which is measurable. By induction on reduced-word
length, every $\alpha_g$ is therefore a bimeasurable nonsingular bijection
between measurable sets. Reversing a reduced word $g=\ell_1\cdots\ell_k$ gives
$\alpha_{g^{-1}}=\vartheta_{\ell_k}^{-1}\circ\cdots\circ\vartheta_{\ell_1}^{-1}
=\alpha_g^{-1}$, which is \eqref{eq:partial-inverse} and shows
$\operatorname{dom}(\alpha_g)=X_{g^{-1}}$.

Write $\rho\leq\kappa$ when $\rho$ is a restriction of $\kappa$, and let
$\vartheta_w$ be the maximal partial composition along an arbitrary word $w$ in
$\mathcal L$. If $w=u\ell\ell^{-1}v$ then
$\vartheta_\ell\circ\vartheta_{\ell^{-1}}\leq\id_X$, so
$\vartheta_w\leq\vartheta_u\circ\vartheta_v$; deleting adjacent inverse pairs
successively gives $\vartheta_w\leq\vartheta_{\operatorname{red}(w)}$. Applied to
the concatenation of the reduced words of $g$ and $h$, this yields
\begin{equation}\tag{$\ddagger$}\label{eq:partial-restriction}
        \alpha_g\circ\alpha_h\leq\alpha_{gh}.
\end{equation}
By \eqref{eq:partial-restriction}, assertion \eqref{eq:partial-composition}
amounts to the domain identity $\operatorname{dom}(\alpha_g\circ\alpha_h)
=\operatorname{dom}(\alpha_h)\cap\operatorname{dom}(\alpha_{gh})$, the
right-hand side being $X_{h^{-1}}\cap X_{h^{-1}g^{-1}}$. One inclusion is
immediate from \eqref{eq:partial-restriction}. Conversely, let $x$ lie in the
right-hand side and put $y=\alpha_h(x)$. Applying
\eqref{eq:partial-restriction} to $gh$ and $h^{-1}$ gives
$\alpha_{gh}\circ\alpha_{h^{-1}}\leq\alpha_g$, and the left-hand side is
defined at $y$ because
$\alpha_{h^{-1}}(y)=x\in\operatorname{dom}(\alpha_{gh})$. Hence
$y\in\operatorname{dom}(\alpha_g)$.

For \eqref{eq:partial-domains}, let $y\in X_{g^{-1}}\cap X_h$ and write
$y=\alpha_h(x)$. Then $\alpha_g\circ\alpha_h$ is defined at $x$, so
$\alpha_g(y)=\alpha_{gh}(x)\in X_g\cap X_{gh}$. Conversely, let
$z\in X_g\cap X_{gh}$, say $z=\alpha_{gh}(x)$, and put $y=\alpha_{g^{-1}}(z)$.
Applying \eqref{eq:partial-restriction} to $g^{-1}$ and $gh$ gives
$\alpha_{g^{-1}}\circ\alpha_{gh}\leq\alpha_h$, hence $y=\alpha_h(x)\in X_h$.
Since also $y\in X_{g^{-1}}$ and $z=\alpha_g(y)$, the reverse inclusion
follows.
\end{proof}

Define the group homomorphism $c:G\to\R$ by $c(a_i)=\log m$ and
$c(b_j)=\log n$, so that $c$ records the energies of the normalized dynamics. The next lemma
shows that $\mu$ is conformal for $c$, in the sense that each $\alpha_g$ scales
the measure by exactly $e^{-c(g)}$.

\begin{lemma}\label{lem:conformal-word}
For every $g\in G$ and every measurable set $E\subseteq X_{g^{-1}}$,
\begin{equation}\label{eq:conformal-word}
        \mu(\alpha_g(E))=e^{-c(g)}\mu(E).
\end{equation}
Equivalently, for every bounded measurable function $F$ supported in $X_g$,
\begin{equation}\label{eq:conformal-change-variables}
        \int_{X_g}F(x)d\mu(x)
        =e^{-c(g)}\int_{X_{g^{-1}}}F(\alpha_g y)d\mu(y).
\end{equation}
\end{lemma}

\begin{proof}
For the positive generators, \eqref{eq:conformal-word} is exactly
\eqref{eq:conformal-basic}. For inverse letters, let $\ell\in\{a_i,b_j\}$ and
let $E\subseteq X_\ell$ be measurable. Applying the formula for $\ell$ to
$\vartheta_{\ell^{-1}}(E)$ gives
$$
        \mu(\vartheta_{\ell^{-1}}(E))
        =e^{c(\ell)}\mu(E)
        =e^{-c(\ell^{-1})}\mu(E).
$$
The formula for a reduced word $g=\ell_1\cdots\ell_k$ now follows by
successively applying the one-letter formula along the maximal partial
composition. Since $c(g)=\sum_{r=1}^k c(\ell_r)$, the product of the
Radon--Nikodym factors is $e^{-c(g)}$. The integral formula follows first for
characteristic functions, then for simple functions, and finally for bounded
measurable functions by uniform approximation.
\end{proof}

Put $D_g=L^\infty(X_g)$, viewed as an ideal in $D=L^\infty(X)$. Since every
$\alpha_g$ is nonsingular, the formula
$$
        \Theta_g:D_{g^{-1}}\to D_g,
        \qquad
        \Theta_g(f)=f\circ\alpha_g^{-1},
$$
is independent of the chosen representative of $f$. By
Lemma~\ref{lem:measurable-partial-action}, the maps $\Theta_g$ form a partial
action of $G$ on $D$ by normal $*$-isomorphisms.

Let $\mathcal B_X^{\rm alg}$ be the algebraic partial crossed product, i.e.\ the
span of symbols $f\delta_g$ with $f\in D_g$, equipped with
\begin{align*}
        (f\delta_g)(h\delta_k)
        &=\Theta_g\left(\Theta_{g^{-1}}(f)h\right)\delta_{gk},\\
        (f\delta_g)^*
        &=\Theta_{g^{-1}}(\overline f)\delta_{g^{-1}}.
\end{align*}
Let $B_X$ be its reduced partial crossed product, equivalently the reduced
cross-sectional algebra of the associated Fell bundle
\cite{Exel1994,ExelBook2017}. On finite sums define
\begin{equation}\label{eq:EX-finite}
        E_X\left(\sum_{g\in F}f_g\delta_g\right)=f_e.
\end{equation}

\begin{proposition}\label{prop:regular-expectation}
The map \eqref{eq:EX-finite} extends to a faithful conditional expectation
$E_X:B_X\to D$.
\end{proposition}

\begin{proof}
Let $\mathcal B=\{B_g\}_{g\in G}$, with $B_g=D_g\delta_g$, be the
semidirect-product Fell bundle of the partial action $\{\Theta_g\}_{g\in G}$, so
that $B_X=C_r^*(\mathcal B)$. The coefficient map at the identity fibre of a
reduced cross-sectional algebra extends to a faithful conditional expectation
$E_X\colon C_r^*(\mathcal B)\to B_e=D$
\cite[Propositions~17.13 and~19.3]{ExelBook2017}. For later use we record the resulting formula on $\mathcal B_X^{\mathrm{alg}}$.
If $z=\sum_{g\in F}f_g\delta_g$, the product $(f_g\delta_g)^*(f_h\delta_h)$
carries the label $g^{-1}h$, which equals $e$ only for $g=h$, in which case it
is $\Theta_{g^{-1}}(|f_g|^2)\delta_e$. Hence
\begin{equation}\label{eq:EX-positive-formula}
        E_X(z^*z)=\sum_{g\in F}\Theta_{g^{-1}}(|f_g|^2),
\end{equation}
a sum of positive elements of $D$.
\end{proof}
The cocycle $c$ induces a dynamics on $B_X$ as follows. For each $t\in\R$,
multiplying the fibre $B_g=D_g\delta_g$ by the unimodular scalar $e^{itc(g)}$
respects the products $B_gB_h\subseteq B_{gh}$ and the involutions
$B_g^*=B_{g^{-1}}$, because $c$ is a group homomorphism. These fibre maps
therefore form an automorphism of the Fell bundle $\mathcal B$, which induces a
$*$-automorphism $\widehat\sigma_t$ of $B_X=C^*_r(\mathcal B)$ satisfying
$$
        \widehat\sigma_t(f\delta_g)=e^{itc(g)}f\delta_g.
$$
The group law $\widehat\sigma_s\circ\widehat\sigma_t=\widehat\sigma_{s+t}$ is
clear on each fibre. On $\mathcal B_X^{\mathrm{alg}}$ each orbit
$t\mapsto\widehat\sigma_t(x)$ is a finite sum of terms $e^{itc(g)}f\delta_g$ and
so is continuous, and this extends to all of $B_X$ because the
$\widehat\sigma_t$ are isometric.

The following proposition is the measurable partial-action analogue of the
usual construction of KMS states from conformal measures, and may be viewed as
a direct verification, in the present model, of the graded-algebra criterion in
\cite[Proposition~1.6(i)]{ExelLaca2003}. For the groupoid and Fell-bundle
formulations of this circle of ideas, see
\cite{Neshveyev2013,AfsarSims2021}.

\begin{proposition}\label{prop:conformal-kms}
The state
$$
        \omega_\mu(x)=\int_X E_X(x)d\mu
$$
on $B_X$ is a $\KMS_1$ state for $\widehat\sigma$.
\end{proposition}

\begin{proof}
First, $\omega_\mu$ is a state: it is positive because
$\omega_\mu(z^*z)=\int_X E_X(z^*z)d\mu\ge0$ by
\eqref{eq:EX-positive-formula} and continuity of $E_X$, and
$\omega_\mu(1)=\mu(X)=1$. Therefore it suffices to verify the KMS identity on
$\mathcal B_X^{\rm alg}$, which is a dense $*$-subalgebra of entire analytic
elements for $\widehat\sigma$, and by linearity it suffices to do so on
monomials. So let $a=f\delta_g$ and $b=h\delta_k$, with $f\in D_g$ and
$h\in D_k$ regarded as extended by zero outside their domains.

Suppose first that $gk\ne e$. Then neither $ab$ nor $b\widehat\sigma_i(a)$ has
an identity coefficient, so $E_X(ab)=E_X(b\widehat\sigma_i(a))=0$ and both
sides of the KMS identity vanish after applying $\omega_\mu$.

Suppose now that $k=g^{-1}$. We compute the two identity coefficients. By the
product formula, the identity coefficient of $(f\delta_g)(h\delta_{g^{-1}})$ is
the function on $X_g$ given by
$$
        x \mapsto f(x)h(\alpha_g^{-1}x),
$$
while, since $\widehat\sigma_i(f\delta_g)=e^{-c(g)}f\delta_g$, the identity
coefficient of $(h\delta_{g^{-1}})\widehat\sigma_i(f\delta_g)$ is the function
on $X_{g^{-1}}$ given by
$$
        y\mapsto e^{-c(g)}h(y)f(\alpha_g y).
$$
The two are related by the change of variables $x=\alpha_g y$, and the
conformality of $\mu$ (Lemma~\ref{lem:conformal-word}, in the form
\eqref{eq:conformal-change-variables}) yields exactly the factor $e^{-c(g)}$.
Hence
$$
        \omega_\mu(ab)
        =\int_{X_g} f(x)h(\alpha_g^{-1}x)d\mu(x)
        =e^{-c(g)}\int_{X_{g^{-1}}} f(\alpha_g y)h(y)d\mu(y)
        =\omega_\mu\left(b\widehat\sigma_i(a)\right).
$$
This proves the KMS identity on $\mathcal B_X^{\rm alg}$, and since this
algebra is an analytic core for $\widehat\sigma$, the identity extends to
$B_X$.
\end{proof}

We have produced a $\KMS_1$ state for $\widehat\sigma$ on $B_X$, and it remains
to pull it back to $\Omn$. Since any KMS state composed with a
unital equivariant $*$-homomorphism is again a KMS state at the same inverse
temperature, by Proposition~\ref{prop:universal} it suffices to exhibit a tame
family in $B_X$ satisfying the defining relations.

\begin{proposition}\label{prop:branch-representation}
The elements $S_i=1_{P_i}\delta_{a_i}$ and $T_j=1_{R_j}\delta_{b_j}$ form a
tame family in $B_X$ satisfying the defining relations of $\Omn$.
Hence there is a unital equivariant representation
$\pi_X:\Omn\to B_X$.
\end{proposition}

\begin{proof}
The initial projection of $S_i$ is $1_Q$, because $\alpha_{a_i}$ has domain
$Q$, and the range projection is $1_{P_i}$. Similarly $T_j^*T_j=1_Q$ and
$T_jT_j^*=1_{R_j}$. Since the $P_i$ are disjoint and the $R_j$ are disjoint,
$$
        S_i^*S_k=0\ (i\ne k),
        \qquad
        T_j^*T_\ell=0\ (j\ne \ell),
$$
and
$$
        \sum_i S_iS_i^*=1_P=\sum_j T_jT_j^*,
        \qquad
        1_P1_Q=0,
        \qquad
        1_P+1_Q=1.
$$
It remains to verify tameness. If $E\subseteq X_g$ and $F\subseteq X_h$ are
measurable, then the crossed-product formulas give
\begin{equation}\label{eq:monomial-product}
(1_E\delta_g)(1_F\delta_h)
 =1_{\alpha_g(\alpha_g^{-1}(E)\cap F)}\delta_{gh},
\qquad
(1_E\delta_g)^*
 =1_{\alpha_g^{-1}(E)}\delta_{g^{-1}}.
\end{equation}
Thus monomials with characteristic-function coefficients are closed under
products and adjoints. Moreover,
\begin{equation}\label{eq:monomial-pi}
(1_E\delta_g)^*(1_E\delta_g)
 =1_{\alpha_g^{-1}(E)}\delta_e,
\qquad
(1_E\delta_g)(1_E\delta_g)^*
 =1_E\delta_e,
\end{equation}
so every such monomial is a partial isometry. Since $S_i,T_j,S_i^*,T_j^*$ are
all monomials of this form, every word in these elements is either zero or a
partial isometry. Hence the family $\{S_i,T_j\}$ satisfies the
Ara--Exel--Katsura relations and is tame. By Proposition~\ref{prop:universal}
there is therefore a unique unital $*$-homomorphism
$\pi_X:\Omn\to B_X$ with $\pi_X(s_i)=S_i$ and $\pi_X(t_j)=T_j$.
Since $S_i$ has cocycle energy $\log m$ and $T_j$ has cocycle energy $\log n$,
this representation is equivariant for the normalized dynamics.
\end{proof}

Before pulling the state back, we record what the model looks like in the
corner. The two families of maps $\theta_i,\eta_j$, which go from $Q$ to $P$,
combine into $mn$ partial bijections \emph{of $Q$}, and the relations
\eqref{eq:corner-relations} appear as two families of partitions of $Q$.

\begin{proposition}\label{prop:corner-model}
Put $Q_{ij}=\theta_i^{-1}(P_i\cap R_j)$ and
$\beta_{ij}=\theta_i^{-1}\circ\eta_j$, a bimeasurable partial bijection of $Q$
with
$$
        \operatorname{dom}\beta_{ij}=\eta_j^{-1}(P_i\cap R_j),
        \qquad
        \operatorname{ran}\beta_{ij}=Q_{ij}.
$$
Then $\beta_{ij}=\alpha_{a_i^{-1}b_j}$ and
\begin{equation}\label{eq:corner-model-generator}
        \pi_X(v_{ij})=1_{Q_{ij}}\delta_{a_i^{-1}b_j},
        \qquad
        \pi_X(v_{ij}v_{ij}^*)=1_{Q_{ij}}\delta_e,
        \qquad
        \pi_X(v_{ij}^*v_{ij})=1_{\operatorname{dom}\beta_{ij}}\delta_e .
\end{equation}
Moreover $\mu(\beta_{ij}(E))=\frac mn\mu(E)$ for measurable
$E\subseteq\operatorname{dom}\beta_{ij}$, and
$$
        \mu(Q_{ij})=m\mu(P_i\cap R_j),
        \qquad
        \mu(\operatorname{dom}\beta_{ij})=n\mu(P_i\cap R_j),
$$
so that for each fixed $i$ the sets $\{Q_{ij}\}_{j}$ partition $Q$, and for
each fixed $j$ the sets $\{\operatorname{dom}\beta_{ij}\}_{i}$ partition $Q$.
\end{proposition}

\begin{proof}
By \eqref{eq:monomial-product}, $S_i^*=1_{\theta_i^{-1}(P_i)}\delta_{a_i^{-1}}
=1_Q\delta_{a_i^{-1}}$, and
$$
        \pi_X(v_{ij})=S_i^*T_j
        =1_{\alpha_{a_i^{-1}}(\alpha_{a_i}(Q)\cap R_j)}\delta_{a_i^{-1}b_j}
        =1_{\theta_i^{-1}(P_i\cap R_j)}\delta_{a_i^{-1}b_j},
$$
which is the first identity in \eqref{eq:corner-model-generator}; the other two
follow from \eqref{eq:monomial-pi} together with
$$\alpha_{a_i^{-1}b_j}=\vartheta_{a_i^{-1}}\circ\vartheta_{b_j}
=\theta_i^{-1}\circ\eta_j=\beta_{ij},$$ whose inverse carries $Q_{ij}$ onto
$\eta_j^{-1}(P_i\cap R_j).$ The scaling factor is
$e^{-c(a_i^{-1}b_j)}=e^{-\varepsilon}=m/n$ by
Lemma~\ref{lem:conformal-word}, and the two displayed measures follow from
\eqref{eq:conformal-basic}. Finally $$\bigsqcup_jQ_{ij}
=\theta_i^{-1}\big(\bigsqcup_j(P_i\cap R_j)\big)=\theta_i^{-1}(P_i)=Q$$ and
$\bigsqcup_i\operatorname{dom}\beta_{ij}=\eta_j^{-1}(R_j)=Q$.
\end{proof}

\begin{theorem}\label{thm:existence-asymmetric}
If $m,n\ge 1$, the normalized system $(\Omn,\sigma)$
admits a $\KMS_1$ state, namely $\varphi_X=\omega_\mu\circ\pi_X$. Equivalently,
writing $\nu=2\mu|_Q$ for the normalized restriction of $\mu$ to $Q$, the
system $(A,\sigma^A)$ admits the $\KMS_1$ state
$$
        \psi_X(x)=\int_Q E_X\big(\pi_X(x)\big)d\nu,
        \qquad x\in A .
$$
\end{theorem}

\begin{proof}
Combine Propositions~\ref{prop:conformal-kms} and
\ref{prop:branch-representation}, noting that a KMS state composed with a
unital equivariant $*$-homomorphism is again a KMS state at the same inverse
temperature. For the second statement, $\pi_X(A)\subseteq 1_QB_X1_Q$ by
Proposition~\ref{prop:corner-model}, so $E_X(\pi_X(x))$ is supported in $Q$ for
$x\in A$ and $\psi_X=2\varphi_X|_A=R_1(\varphi_X)$ is a $\KMS_1$ state of
$(A,\sigma^A)$ by \eqref{eq:R1}.
\end{proof}

\subsection{The polytope and non-uniqueness}
\label{sec:nonunique}

Throughout this subsection, we assume only that $m,n\ge 2$.
Let
\begin{equation}\label{eq:transport}
\mathcal{P}:=\left\{ C=(c_{ij})_{ij}\in M_{m\times n}(\mathbb{R}):
        c_{ij}\ge0, \ 
        \sum_j c_{ij}=\frac1m, \
        \sum_i c_{ij}=\frac1n \right\},
\end{equation}
and note that the entries of such a matrix sum to one. We shall realize
$c_{ij}/2$ as the measure of the overlap between the $i$th $s$-range and the
$j$th $t$-range. Accordingly, set
$$
        f_{ij}=t_jt_j^*s_is_i^*t_jt_j^* .
$$
This element is positive, and in the tame algebra it coincides with the overlap
projection $s_is_i^*t_jt_j^*$, the two range projections commuting (see
Remark~\ref{rem:inverse-semigroup}). Moreover, by Lemma~\ref{lem:corner-elementary},
\begin{equation}\label{eq:f-vs-v}
        f_{ij}=s_i(v_{ij}v_{ij}^*)s_i^*,
        \qquad
        v_{ij}v_{ij}^*=s_i^*f_{ij}s_i .
\end{equation}

\begin{lemma}\label{lem:dictionary}
Let $\varphi$ be a $\KMS_1$ state of $(\Omn,\sigma)$, let
$\psi=2\varphi|_A$ be the corresponding $\KMS_1$ state of $(A,\sigma^A)$, and
put $c_{ij}=2\varphi(f_{ij})$. Then
\begin{equation}\label{eq:dictionary}
        \psi\big(v_{ij}v_{ij}^*\big)=mc_{ij},
        \qquad
        \psi\big(v_{ij}^*v_{ij}\big)=nc_{ij} .
\end{equation}
\end{lemma}

\begin{proof}
Since $\sigma_z(s_i^*)=m^{-iz}s_i^*$ we have $\sigma_i(s_i^*)=ms_i^*$, so the
$\KMS_1$ identity applied to the pair $a=s_i^*$, $b=t_jt_j^*s_i$ gives
$$
        \varphi\big(v_{ij}v_{ij}^*\big)
        =\varphi\big(s_i^*t_jt_j^*s_i\big)
        =\varphi\big(t_jt_j^*s_i\sigma_i(s_i^*)\big)
        =m\varphi\big(t_jt_j^*s_is_i^*\big)=m\varphi(f_{ij}),
$$
and multiplying by $2$ gives the first identity. The second follows in the same
way from $\sigma_i(t_j^*)=nt_j^*$, applied to the pair $a=t_j^*$,
$b=s_is_i^*t_j$.
\end{proof}

\begin{lemma}\label{lem:realize-transport}
For every $C\in\mathcal P$ there are non-atomic probability data as in
Section~\ref{sec:regular-models} with $\mu(P_i\cap R_j)=\tfrac12c_{ij}$;
equivalently, with $\nu(Q_{ij})=mc_{ij}$ in the notation of
Proposition~\ref{prop:corner-model}.
\end{lemma}

\begin{proof}

Let $P$ be a disjoint union of measurable sets $P_{ij}$ with
$\mu(P_{ij})=c_{ij}/2$, and put
$$
        P_i=\bigsqcup_j P_{ij},
        \qquad
        R_j=\bigsqcup_i P_{ij}.
$$
The row and column conditions in \eqref{eq:transport} give $\mu(P_i)=1/(2m)$
and $\mu(R_j)=1/(2n)$. Since $Q$ is non-atomic of measure $1/2$, there are
measurable isomorphisms $Q\to P_i$ and $Q\to R_j$ with the scaling factors
required in \eqref{eq:conformal-basic}. The last assertion is
$\nu(Q_{ij})=2m\mu(P_i\cap R_j)$ from
Proposition~\ref{prop:corner-model}.
\end{proof}

\begin{proposition}\label{prop:distinguish-C}
Let $\varphi_C$ be a $\KMS_1$ state obtained from a branching model with
first-level intersection matrix $C$, and let $\psi_C=2\varphi_C|_A$. Then
$$
        \varphi_C(f_{ij})=\tfrac12c_{ij},
        \qquad
        \psi_C\big(v_{ij}v_{ij}^*\big)=mc_{ij}.
$$
\end{proposition}

\begin{proof}
Under $\pi_X$ the projection $s_is_i^*$ becomes $1_{P_i}$ and $t_jt_j^*$
becomes $1_{R_j}$, so that $\pi_X(f_{ij})=1_{R_j}1_{P_i}1_{R_j}=1_{P_i\cap
R_j}$. Since $\varphi_C=\omega_\mu\circ\pi_X$ and $E_X$ is the identity on the
diagonal,
$$
        \varphi_C(f_{ij})=\int_X 1_{P_i\cap R_j}d\mu=\mu(P_i\cap R_j)
        =\tfrac12c_{ij}.
$$
The second formula is Lemma~\ref{lem:dictionary}; alternatively, it is
$\psi_C(v_{ij}v_{ij}^*)=\nu(Q_{ij})=mc_{ij}$ directly from
\eqref{eq:corner-model-generator}.
\end{proof}

The values of a state on the projections $f_{ij}$ are defined for \emph{every}
$\KMS_1$ state, not only for those coming from a branching model, and by
Lemma~\ref{lem:dictionary} they are constrained by the two relations
\eqref{eq:corner-relations}. This yields the following.

\begin{proposition}\label{prop:Phi}
The map $\Phi$ from the $\KMS_1$ simplex of $(\Omn,\sigma)$ to
$\mathcal P$ given by $\Phi(\varphi)=\left(2\varphi(f_{ij})\right)_{ij}$ is
well defined, affine, weak-$*$ continuous and surjective. Moreover
$\dim\mathcal P=(m-1)(n-1)$.
\end{proposition}

\begin{proof}
Positivity of $f_{ij}$ gives $2\varphi(f_{ij})\ge0$. Write $\psi=2\varphi|_A$
and $c_{ij}=2\varphi(f_{ij})$. Summing the first identity of
\eqref{eq:dictionary} over $j$, and the second over $i$, the relations
\eqref{eq:corner-relations} and $\psi(q)=1$ give
$$
        m\sum_j c_{ij}=\psi\Big(\sum_j v_{ij}v_{ij}^*\Big)=\psi(q)=1,
        \qquad
        n\sum_i c_{ij}=\psi\Big(\sum_i v_{ij}^*v_{ij}\Big)=\psi(q)=1,
$$
so $\Phi(\varphi)\in\mathcal P$. Each coordinate
$\varphi\mapsto2\varphi(f_{ij})$ is evaluation at a fixed element, hence a
weak-$*$ continuous affine functional; so is $\Phi$. Surjectivity follows from
Lemma~\ref{lem:realize-transport} together with
Proposition~\ref{prop:distinguish-C}, since every $C\in\mathcal P$ equals
$\Phi(\varphi_C)$ for a suitable branching model.

For the dimension, the $mn$ entries are subject to the $m+n$ equations in
\eqref{eq:transport}, one of which is redundant because both families prescribe
total mass one. The uniform matrix $c_{ij}=1/(mn)$ satisfies all the
inequalities strictly, so $\mathcal P$ has nonempty relative interior in the
resulting affine subspace and
$\dim\mathcal P=mn-(m+n-1)=(m-1)(n-1)$.
\end{proof}

We are now able to prove that $\KMS_1$ states are non-unique in the asymmetric case $m\neq n$. We emphasize, however, that this assumption is not required in the following result. Thus, the result also recovers the non-uniqueness in the diagonal case, which is already known.

\begin{corollary}\label{cor:nonuniqueness-asymmetric}
If $m,n\ge2$, the normalized system $(\Omn,\sigma)$
has more than one $\KMS_1$ state.
\end{corollary}

\begin{proof}
By Proposition~\ref{prop:Phi} the $\KMS_1$ simplex surjects onto $\mathcal P$,
which has positive dimension $(m-1)(n-1)$ and therefore more than one point.
\end{proof}

\subsection{Fibres over the vertices of the polytope}
\label{sec:fibres}
We continue to assume that $m,n \ge 2$. We first show that the map $\Phi$ from Proposition~\ref{prop:Phi} is not injective, so that distinct $\KMS_1$ states can have the same first-level intersection matrix. Recall that a vertex of a polytope is a point that cannot be expressed as a nontrivial convex combination of two distinct points in the polytope. We focus on the fibres over the vertices of $\mathcal P$, showing that each such fibre contains an extremal, and hence factorial, $\KMS_1$ state. Moreover, we show that each such fibre is affinely homeomorphic to the $\KMS_1$ simplex of a quotient of the corner algebra $A$ by an ideal generated by a certain collection of generators. The description of the remaining fibres is part of the classification problem left open in the introduction.

\begin{proposition}\label{prop:notinjective}
The map $\Phi: K_1(\Omn, \sigma)\to \mathcal{P}$ of Proposition~\ref{prop:Phi} is not
injective.
\end{proposition}

\begin{proof}
When $(m,n)=(2,2)$ non-injectivity is already visible in
Example~\ref{ex:O22}. Indeed, the states $\varphi_{1,0}$ and $\varphi_{1,\pi}$
there are distinct, since they take the values $\pm\tfrac12$ at $s_2^{*}t_2$,
while in the representation $\pi_{1,U}$ one has $\pi_{1,U}(v_{12})=V_1^{*}V_2U_2=0$
and likewise $\pi_{1,U}(v_{21})=0$, so that both states vanish on $f_{12}$ and
$f_{21}$ and
$$
        \Phi(\varphi_{1,0})=\Phi(\varphi_{1,\pi})
        =\begin{pmatrix}1/2&0\\0&1/2\end{pmatrix}.
$$
We may thus assume that $m\ge3$ or $n\ge3$. In that case non-injectivity follows from the fact that
polytope $\mathcal P$ is not a simplex, whereas the $\KMS_1$ states always
form one.

Write $d=\dim\mathcal P=(m-1)(n-1)$ and recall that a simplex of finite
dimension $d$ has exactly $d+1$ facets, i.e. maximal proper faces. We claim that $\mathcal P$ has $mn$
facets, namely
$$
        F_{uv}=\{C\in\mathcal P:\ c_{uv}=0\},\qquad 1\le u\le m,\ 1\le v\le n.
$$
Each $F_{uv}$ is a face, being the intersection of $\mathcal P$ with the
hyperplane $\{c_{uv}=0\}$. To see that it is a facet we exhibit a point of
$\mathcal P$ at which $c_{uv}=0$ and every other entry is strictly positive;
such a point lies in the relative interior of $F_{uv}$, hence
$\dim F_{uv}=d-1$.

We assume that $n\ge3$, transposing if instead
$m\ge3$. Fix $k\ne u$ and distinct $\ell_1,\ell_2\ne v$, and set
$\omega=1/(mn)$. Starting from the uniform matrix $c_{ij}=\omega$, add for
$r=1,2$ the perturbation supported on $(u,v),(u,\ell_r),(k,v),(k,\ell_r)$ given
by
$$
        c_{uv}\mapsto c_{uv}-\tfrac\omega2,\quad
        c_{u\ell_r}\mapsto c_{u\ell_r}+\tfrac\omega2,\quad
        c_{kv}\mapsto c_{kv}+\tfrac\omega2,\quad
        c_{k\ell_r}\mapsto c_{k\ell_r}-\tfrac\omega2 .
$$
Each perturbation leaves all row and column sums unchanged, so the resulting
matrix $\widehat{C}$ lies in the affine hull of $\mathcal P$. Its entries are
$$
        \widehat{c}_{uv}=0,\quad
        \widehat{c}_{u\ell_1}=\widehat{c}_{u\ell_2}=\tfrac32\omega,\quad
        \widehat{c}_{kv}=2\omega,\quad
        \widehat{c}_{k\ell_1}=\widehat{c}_{k\ell_2}=\tfrac\omega2,
$$
all remaining entries being equal to $\omega$. Hence $\widehat{C}\in\mathcal P$
and $\widehat{c}_{ij}>0$ for $(i,j)\ne(u,v)$, as required. In particular the
$F_{uv}$ are pairwise distinct, since $F_{uv}$ is the only one of them
containing $\widehat{C}$. As $m+n>2$ we have $mn>d+1$, so $\mathcal P$ is not a
simplex.

Suppose now that $\Phi$ were injective. Being a continuous bijection from a
weak-$*$ compact space onto a Hausdorff space, it would be a homeomorphism, and
being affine it would carry the $\KMS_1$ simplex affinely onto $\mathcal P$.
Since the set of $\KMS_1$ states is a simplex, so would be $\mathcal P$; which
we have just excluded.
\end{proof}

We recall two standard notions. A $\KMS_\beta$ state is extremal if it
is an extreme point of the convex set of $\KMS_\beta$ states. A state $\varphi$
is a factor state if the von Neumann algebra
$\pi_\varphi(\Omn)''$ generated by its GNS representation has
trivial centre, and two states are disjoint if their GNS representations
have no nonzero unitarily equivalent subrepresentations. By
\cite[Theorem~5.3.30]{BR2}, for a unital $C^*$-dynamical system the extremal
$\KMS_\beta$ states are exactly the factor states, and two distinct ones are
disjoint.

\begin{theorem}\label{thm:vertices}
For every vertex $V$ of $\mathcal P$ the fibre $\Phi^{-1}(V)$ is a closed face
of the $\KMS_1$ simplex, and hence contains an extremal $\KMS_1$ state
$\varphi_V$ with $\Phi(\varphi_V)=V$. Consequently $\Omn$ admits a
family of pairwise disjoint factorial $\KMS_1$ states, one for each vertex of
$\mathcal P$.
\end{theorem}

\begin{proof}
Write $K_1$ for the simplex of $\KMS_1$ states and fix a vertex $V$ of
$\mathcal P$. The fibre $F=\Phi^{-1}(V)$ is nonempty by surjectivity of $\Phi$,
weak-$*$ closed by its continuity, and convex since $\Phi$ is affine.

We claim $F$ is a face of $K_1$. Let $\varphi\in F$ and suppose
$\varphi=\lambda\varphi_1+(1-\lambda)\varphi_2$ with
$\varphi_1,\varphi_2\in K_1$ and $0<\lambda<1$. Applying $\Phi$ gives
$V=\lambda\Phi(\varphi_1)+(1-\lambda)\Phi(\varphi_2)$, a convex combination of
two points of $\mathcal P$; as $V$ is an extreme point of $\mathcal P$ we get
$\Phi(\varphi_1)=\Phi(\varphi_2)=V$, that is, $\varphi_1,\varphi_2\in F$.

Being a nonempty weak-$*$ compact convex set, $F$ has an extreme point
$\varphi_V$ by the Krein--Milman theorem, and an extreme point of a face is an
extreme point of the whole set. Hence $\varphi_V$ is an extremal $\KMS_1$ state
with $\Phi(\varphi_V)=V$. By \cite[Theorem~5.3.30]{BR2} these states are
factorial, and distinct ones are disjoint; choosing one $\varphi_V $ for each
vertex gives states that are pairwise distinct, since $\Phi(\varphi_V)=V$
separates them.
\end{proof}

Theorem~\ref{thm:vertices} therefore yields at least as many extremal factorial
$\KMS_1$ states as there are vertices of $\mathcal P$. Since the marginals
are uniform in the present setting, $\mathcal P$ is, up to scaling, the
central transportation polytope of size $m\times n$. The
combinatorics of the vertices of transportation polytopes is well
studied; see, for example,~\cite[Section~2.2]{DeLoeraKim2014}.
In the diagonal case $m=n=r$, the vertex set is particularly
simple and will be described explicitly in Proposition~\ref{prop:diagonal-fibres}.

\begin{example}\label{ex:O23}
For $\mathcal O_{2,3}$, equation \eqref{eq:partition} reads
$2^{1-\beta}=3^{1-\beta}$, so only $\beta=1$ is possible, and the polytope
$\mathcal P$ has dimension $2$. Writing the first row as $(x,y,z)$ identifies
$\mathcal P$ with $\{x+y+z=\tfrac12,\ 0\le x,y,z\le\tfrac13\}$, a hexagon with
six vertices, obtained from the triangle $x+y+z=\tfrac12$ by truncating its
three corners; see Figure~\ref{fig:hexagon}.

\begin{figure}[H]
\centering
\begin{tikzpicture}[
scale=0.5,
  x={(-0.75cm,-0.45cm)}, y={(1.25cm,0cm)}, z={(0cm,1.20cm)}, >=Latex,
  cube/.style={gray!55, thin},
  vert/.style={circle, fill=black, inner sep=1.3pt},
  ctr/.style={circle, draw=black, fill=white, inner sep=1.2pt}]
  \foreach \a/\b in {{0,0,0}/{2,0,0}, {0,0,0}/{0,2,0}, {0,0,0}/{0,0,2},
                     {2,2,2}/{0,2,2}, {2,2,2}/{2,0,2}, {2,2,2}/{2,2,0},
                     {2,0,0}/{2,2,0}, {2,0,0}/{2,0,2}, {0,2,0}/{2,2,0},
                     {0,2,0}/{0,2,2}, {0,0,2}/{2,0,2}, {0,0,2}/{0,2,2}}
     \draw[cube] (\a) -- (\b);
  \fill[gray!30, opacity=0.75]
     (2,1,0)--(2,0,1)--(1,0,2)--(0,1,2)--(0,2,1)--(1,2,0)--cycle;
  \draw[thick] (2,1,0)--(2,0,1)--(1,0,2)--(0,1,2)--(0,2,1)--(1,2,0)--cycle;
  \draw[dashed, black!55] (2,1,0)--(0,1,2);
  \foreach \p in {{2,1,0},{2,0,1},{1,0,2},{0,1,2},{0,2,1},{1,2,0}}
     \node[vert] at (\p) {};
  \node[ctr] at (1,1,1) {};
  \node[font=\footnotesize, anchor=north east] at (2,1,0) {$V$};
  \node[font=\footnotesize, anchor=south west] at (0,1,2) {$V'$};
  \draw[->, gray] (0,0,0) -- (2.75,0,0)
     node[anchor=north east, font=\footnotesize, black] {$c_{11}$};
  \draw[->, gray] (0,0,0) -- (0,2.75,0)
     node[anchor=west, font=\footnotesize, black] {$c_{12}$};
  \draw[->, gray] (0,0,0) -- (0,0,2.75)
     node[anchor=south, font=\footnotesize, black] {$c_{13}$};
\end{tikzpicture}
\caption{The transportation polytope $\mathcal P$ for $\mathcal O_{2,3}$,
realized in the first row $(c_{11},c_{12},c_{13})$. The section is a regular
hexagon whose six vertices are midpoints of cube edges, and whose six edges are
the facets $\{c_{ij}=0\}$ of Proposition~\ref{prop:notinjective}.}
\label{fig:hexagon}
\end{figure}
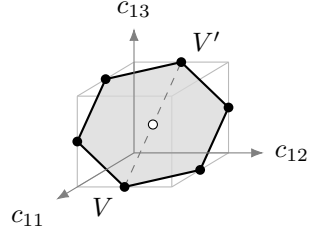
Two opposite vertices are
$$
        V=
        \begin{pmatrix}
        1/3 & 1/6 & 0\\
        0   & 1/6 & 1/3
        \end{pmatrix},
        \qquad
        V'=
        \begin{pmatrix}
        0 & 1/6 & 1/3\\
        1/3 & 1/6 & 0
        \end{pmatrix},
$$
with midpoint the uniform matrix with all entries $1/6$. A concrete interval
model for $V$ is obtained by taking $X=[0,1)$, $Q=[0,1/2)$, $P=[1/2,1)$ and
splitting
$$
        P_1=[1/2,3/4),\ P_2=[3/4,1);
        \qquad
        R_1=[1/2,2/3),\ R_2=[2/3,5/6),\ R_3=[5/6,1),
$$
so that $\mu(P_1\cap R_1)=1/6$, $\mu(P_1\cap R_2)=1/12$, $\mu(P_1\cap R_3)=0$,
and symmetrically in the second row. The associated state satisfies
$\varphi(f_{11})=1/6$, $\varphi(f_{12})=1/12$ and $\varphi(f_{13})=0$, whereas
the state associated with the uniform matrix gives $\varphi(f_{11})=1/12$. In
the corner, with $\theta_1(x)=\tfrac12+\tfrac x2$, the sets of
Proposition~\ref{prop:corner-model} are
$$
        Q_{11}=[0,\tfrac13),\qquad Q_{12}=[\tfrac13,\tfrac12),
        \qquad Q_{13}=\emptyset,
$$
of $\nu$-measures $\tfrac23,\tfrac13,0$, that is $mc_{1j}$ as in
Lemma~\ref{lem:realize-transport}. By Theorem~\ref{thm:vertices},
$\mathcal O_{2,3}$ carries at least six pairwise disjoint extremal factorial
$\KMS_1$ states, one for each vertex of the polytope.
\end{example}

In the next result our goal is to identify the fibres over the vertices of
$\mathcal P$ with $\KMS_1$ simplices of quotients of $A$. We shall use a
standard combinatorial description of these vertices. Let $K_{m,n}$ be the
complete bipartite graph whose two sides are the vertices $P_1,\dots,P_m$ and
$R_1,\dots,R_n$, and associate to a matrix $C\in\mathcal P$ the subgraph having
an edge between $P_i$ and $R_j$ precisely when $c_{ij}>0$. It is well known
that $C$ is a vertex of $\mathcal P$ if and only if its associated subgraph is
a forest, that is, contains no cycle; see e.g.\ the survey
\cite[Lemma~2.9]{DeLoeraKim2014}.

\begin{example}
For $(m,n)=(2,3)$, the matrix $V$ of Example~\ref{ex:O23} has associated
subgraph
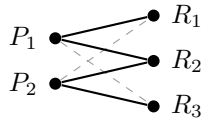
\begin{center}
\begin{tikzpicture}[
    scale=0.5,
    dot/.style={circle, fill=black, inner sep=1.6pt},
    every label/.style={font=\footnotesize}
]
  \node[dot, label=left:{$P_1$}]  (P1) at (0, 1.2) {};
  \node[dot, label=left:{$P_2$}]  (P2) at (0, 0)   {};
  \node[dot, label=right:{$R_1$}] (R1) at (2.6, 1.8) {};
  \node[dot, label=right:{$R_2$}] (R2) at (2.6, 0.6) {};
  \node[dot, label=right:{$R_3$}] (R3) at (2.6,-0.6) {};
  \draw[dashed, black!40] (P1) -- (R3);
  \draw[dashed, black!40] (P2) -- (R1);
  \draw[thick] (P1) -- (R1);
  \draw[thick] (P1) -- (R2);
  \draw[thick] (P2) -- (R2);
  \draw[thick] (P2) -- (R3);
\end{tikzpicture}
\end{center}
where the dashed edges correspond to the vanishing entries. Being a tree, it
follows that $V$ is a vertex of $\mathcal P$. The uniform matrix, by contrast,
has all six edges present and is not a vertex.
\end{example}

\begin{lemma}\label{lem:vertex}
Assume $m,n\ge2$. Let $V=(c_{ij})$ be a vertex of $\mathcal P$ and
let $Z=\{(i,j):c_{ij}=0\}$ be the complement of its support. Let $J_Z$ be the
closed two-sided ideal of $A$ generated by $\{v_{ij}:(i,j)\in Z\}$, and $I_Z$
the closed two-sided ideal of $\Omn$ generated by
$\{f_{ij}:(i,j)\in Z\}$. Then:
\begin{enumerate}[label=\textup{(\roman*)}]
\item $J_Z$ is $\sigma^A$-invariant and $I_Z$ is $\sigma$-invariant; the
respective dynamics descend to strongly continuous actions $\widetilde\sigma^A$ on
$A/J_Z$ and $\widetilde\sigma$ on $\Omn/I_Z$;
\item $\Xi(I_Z)=M_{m+1}(J_Z)$; in particular $J_Z=I_Z\cap A$ and
$\Omn/I_Z\cong M_{m+1}(A/J_Z)$;
\item for a $\KMS_1$ state $\varphi$ of $(\Omn,\sigma)$, with
$\psi=2\varphi|_A$, the following are equivalent: $\Phi(\varphi)=V$;\ \
$\varphi$ vanishes on $I_Z$;\ \ $\psi$ vanishes on $J_Z$;
\item $J_Z$ is a proper nonzero ideal of $A$.
\end{enumerate}
\end{lemma}

\begin{proof}
\begin{enumerate}[label=\textup{(\roman*)}]
\item By \eqref{eq:corner-dynamics} each generator of $J_Z$ satisfies
$\sigma^A_t(v_{ij})=e^{it\varepsilon}v_{ij}$, a scalar multiple of itself, so
$\sigma^A_t(J_Z)=J_Z$; and $\sigma_t(f_{ij})=f_{ij}$ because
$\sigma_t(s_is_i^*)=s_is_i^*$ and $\sigma_t(t_jt_j^*)=t_jt_j^*$, so
$\sigma_t(I_Z)=I_Z$. In both cases the quotient action is a one-parameter group
of $*$-automorphisms, strongly continuous because the quotient map is
contractive.

\item Since $v_{ij}$ is a partial isometry, $v_{ij}=(v_{ij}v_{ij}^*)v_{ij}$, so
$v_{ij}$ and $v_{ij}v_{ij}^*$ generate the same closed ideal; hence $J_Z$ is
generated by $\{v_{ij}v_{ij}^*:(i,j)\in Z\}$. By \eqref{eq:f-vs-v} and
Proposition~\ref{prop:matrix-picture},
$$
        \Xi(f_{ij})_{ab}=u_a^*f_{ij}u_b
        =\delta_{ai}\delta_{bi}s_i^*f_{ij}s_i
        =\delta_{ai}\delta_{bi}v_{ij}v_{ij}^*,
$$
that is $\Xi(f_{ij})=E_{ii}\otimes v_{ij}v_{ij}^*$.
The closed ideals of $M_{m+1}(A)$ are exactly the $M_{m+1}(J)$ with $J$ a
closed ideal of $A$, and the ideal generated by $E_{ii}\otimes x$ is
$M_{m+1}(\langle x\rangle)$. Hence $\Xi(I_Z)=M_{m+1}(J_Z)$, and intersecting
with $E_{00}\otimes A$ gives $J_Z=I_Z\cap A$.

\item Suppose first that $\psi(v_{ij}v_{ij}^*)=0$ for every $(i,j)\in Z$; we
show that $\psi$ vanishes on $J_Z$. Fix such an $(i,j)$, write
$e=v_{ij}v_{ij}^*$ and let $(H_\psi,\pi_\psi,\xi)$ be the GNS triple of $\psi$.
As $e$ is a projection, $\|\pi_\psi(e)\xi\|^2=\psi(e)=0$, hence
$\pi_\psi(e)\xi=0$ and
\begin{equation}\tag{$\dagger\dagger$}\label{eq:e-left}
        \psi(eb)=\langle\pi_\psi(b)\xi,\pi_\psi(e)\xi\rangle=0
        \qquad (b\in A).
\end{equation}
For $a,b$ entire analytic, the element $eb$ is entire analytic and
$\sigma^A_i(e)=e$, so the $\KMS_1$ identity applied to the pair $a$, $eb$ gives
$$
        \psi(aeb)=\psi\left(eb\sigma^A_i(a)\right)
        =\psi\left(e(b\sigma^A_i(a))\right)=0
$$
by \eqref{eq:e-left}. By density and continuity $\psi(aeb)=0$ for all $a,b$,
and since the closed ideal generated by $e$ is the closed span of such
products, $\psi$ vanishes on $J_Z$.

Put $\mathcal P_Z=\{C\in\mathcal P: c_{ij}=0\text{ for }(i,j)\in Z\}$. If
$C\in\mathcal P_Z$, its support is contained in that of $V$, which is a forest
because $V$ is a vertex; hence $C$ is itself a vertex. So every point of the
compact convex set $\mathcal P_Z$ is extreme, forcing $\mathcal P_Z=\{V\}$.

Now the three conditions. If $\Phi(\varphi)=V$ then, by
Lemma~\ref{lem:dictionary}, $\psi(v_{ij}v_{ij}^*)=mc_{ij}=0$ for
$(i,j)\in Z$, so $\psi$ vanishes on $J_Z$ by the previous paragraph. If $\psi$
vanishes on $J_Z$ then $\varphi$ vanishes on $I_Z$: by \eqref{eq:reconstruction}
at $\beta=1$,
$$
        \varphi(x)=\tfrac12\Big(\psi(qxq)+m^{-1}\sum_i\psi(s_i^*xs_i)\Big),
$$
and for $x\in I_Z$ both $qxq$ and $s_i^*xs_i$ lie in $I_Z\cap A=J_Z$ by (ii).
Finally, if $\varphi$ vanishes on $I_Z$ then in particular
$\varphi(f_{ij})=0$ for $(i,j)\in Z$, hence
$\Phi(\varphi)\in\mathcal P_Z=\{V\}$.

\item Since $V$ is a vertex, its support is a forest in $K_{m,n}$ and therefore
has at most $m+n-1$ edges, so that $|Z|\ge mn-(m+n-1)=(m-1)(n-1)\ge1$; in
particular $Z\ne\emptyset$. Moreover each $v_{ij}$ is nonzero: by
Lemma~\ref{lem:realize-transport} the uniform matrix $c_{ij}=1/(mn)$ is
realized by a branching model, and the associated state satisfies
$\psi(v_{ij}v_{ij}^*)=1/n>0$ by Proposition~\ref{prop:distinguish-C}. Hence
$J_Z$ is nonzero. Finally, taking $\varphi$ with $\Phi(\varphi)=V$, which
exists by Proposition~\ref{prop:Phi}, part~(iii) gives $\psi|_{J_Z}=0$ while
$\psi(q)=1$, so $J_Z$ is proper.
\end{enumerate}
\end{proof}

\begin{theorem}\label{thm:fibres-vertex}
In the notation of Lemma~\ref{lem:vertex}, let
$\pi_A\colon A\to A/J_Z$ be the quotient map. Then
$$
        K_1\left(A/J_Z,\widetilde\sigma^A\right)
        \xrightarrow{\ \cong\ }
        \Phi^{-1}(V),
        \qquad
        \widetilde\psi\longmapsto R_1^{-1}\big(\widetilde\psi\circ\pi_A\big),
$$
is an affine weak-$*$ homeomorphism. The same holds with $(A,J_Z)$ replaced by
$(\Omn,I_Z)$ and the induced dynamics $\widetilde\sigma$.
\end{theorem}

\begin{proof}
By \eqref{eq:R1} and Lemma~\ref{lem:vertex}(iii), $R_1$ restricts to an affine
weak-$*$ homeomorphism from $\Phi^{-1}(V)$ onto the set of $\KMS_1$ states of
$(A,\sigma^A)$ vanishing on $J_Z$. It therefore suffices to identify the latter
with $K_1(A/J_Z,\widetilde\sigma^A)$ through $\widetilde\psi\mapsto\widetilde\psi\circ\pi_A$.

If $\widetilde\psi$ is $\KMS_1$ for $\widetilde\sigma^A$, then $\widetilde\psi\circ\pi_A$
is a state vanishing on $J_Z$, and it is $\KMS_1$ for $\sigma^A$ because
$\pi_A$ is unital and equivariant. Conversely, a $\KMS_1$ state $\psi$ of
$(A,\sigma^A)$ vanishing on $J_Z$ factors as $\psi=\widetilde\psi\circ\pi_A$ for a
unique state $\widetilde\psi$ of $A/J_Z$. To see that $\widetilde\psi$ is $\KMS_1$, let
$A_0$ be the $*$-algebra generated algebraically by the $v_{ij}$; by
Proposition~\ref{prop:corner-generators} it is dense in $A$, and it is spanned
by words in the $v_{ij}$ and their adjoints, each of which is homogeneous for
$\sigma^A$ by \eqref{eq:corner-dynamics} and hence entire analytic. Thus
$\pi_A(A_0)$ is a dense, $\widetilde\sigma^A$-invariant $*$-subalgebra of $A/J_Z$
consisting of entire analytic elements, and the $\KMS_1$ identity for
$\widetilde\psi$ on $\pi_A(A_0)$ is exactly the $\KMS_1$ identity for $\psi$ on
$A_0$. Hence $\widetilde\psi\in K_1(A/J_Z,\widetilde\sigma^A)$. The map is injective
because $\pi_A$ is surjective, and affinity and weak-$*$ continuity are
immediate; a continuous bijection between weak-$*$ compact sets is a
homeomorphism.

The statement for $\Omn/I_Z$ follows in the same way from
Lemma~\ref{lem:vertex}(iii), or from Lemma~\ref{lem:vertex}(ii) and the
identification $\Omn/I_Z\cong M_{m+1}(A/J_Z)$.
\end{proof}

\begin{example}\label{ex:fibre}
Continue with $\mathcal O_{2,3}$ and the vertex $V$ of Example~\ref{ex:O23}, so
that $Z=\{(1,3),(2,1)\}$ and $J_Z$ is generated by $v_{13}$ and $v_{21}$.
Writing the images in $A/J_Z$ with the same letters, the relations
\eqref{eq:corner-relations} become
$$
        v_{11}v_{11}^*+v_{12}v_{12}^*=q,
        \qquad
        v_{22}v_{22}^*+v_{23}v_{23}^*=q,
$$
$$
        v_{11}^*v_{11}=q,
        \qquad
        v_{12}^*v_{12}+v_{22}^*v_{22}=q,
        \qquad
        v_{23}^*v_{23}=q .
$$
In particular $v_{11}$ and $v_{23}$ have become isometries. Put
$e=v_{11}v_{11}^*$ and $f=v_{23}v_{23}^*$; by
Remark~\ref{rem:inverse-semigroup} these are commuting projections.

Let $\widetilde\psi\in K_1(A/J_Z,\widetilde\sigma^A)$. Then
$\widetilde\psi(v_{11}^*v_{11})=\widetilde\psi(q)=1$, and
Lemma~\ref{lem:basic-kms} with $\beta=1$ and $E=\varepsilon=\log(3/2)$ gives
$\widetilde\psi(e)=\tfrac23$, and likewise $\widetilde\psi(f)=\tfrac23$. Since
$e$ and $f$ commute, $ef$ is a projection with $ef\le e$, and
$e\vee f=e+f-ef\le q$. Applying $\widetilde\psi$ to these two relations gives
\begin{equation}\tag{$\ddagger\ddagger$}\label{eq:overlap}
        \tfrac13=\widetilde\psi(e)+\widetilde\psi(f)-\widetilde\psi(q)
        \le\widetilde\psi(ef)\le\widetilde\psi(e)=\tfrac23 .
\end{equation}
Every value in this interval occurs. Fix $\delta\in[0,\tfrac16]$ and take the
interval model of Example~\ref{ex:O23}, that is $X=[0,1)$, $Q=[0,\tfrac12)$,
$P=[\tfrac12,1)$ with
$$
        P_1=[\tfrac12,\tfrac34),\ P_2=[\tfrac34,1);
        \qquad
        R_1=[\tfrac12,\tfrac23),\ R_2=[\tfrac23,\tfrac56),\
        R_3=[\tfrac56,1),
$$
let $\theta_1(x)=\tfrac12+\tfrac x2$ and let the $\eta_j$ be affine, but
replace $\theta_2$ by the piecewise affine bijection
$\theta_2^{(\delta)}\colon Q\to P_2$ of constant slope $\tfrac12$, as in
Figure~\ref{fig:theta-delta}.

\begin{figure}[ht]
\centering
\def\dd{0.0833}
\pgfmathsetmacro{\dhalf}{\dd/2}
\pgfmathsetmacro{\xb}{\dd+1/3}
\pgfmathsetmacro{\yR}{1/12}
\begin{minipage}{0.47\textwidth}
\[
\theta_2^{(\delta)}(x)=
\begin{cases}
\tfrac34+\tfrac x2, & x\in[0,\delta),\\[3pt]
\tfrac56+\tfrac{x-\delta}{2}, & x\in[\delta,\delta+\tfrac13),\\[3pt]
\tfrac34+\tfrac\delta2+\tfrac{x-\delta-1/3}{2}, &
x\in[\delta+\tfrac13,\tfrac12).
\end{cases}
\]
\end{minipage}%
\begin{minipage}{0.44\textwidth}
\centering
\begin{tikzpicture}[xscale=6.2,yscale=9,>=stealth,
  clo/.style={fill,circle,inner sep=0.9pt},
  ope/.style={draw,fill=white,circle,inner sep=0.85pt,line width=0.4pt}]
  \draw[->] (-0.015,0) -- (0.55,0) node[right,font=\scriptsize] {$x$};
  \draw[->] (0,-0.015) -- (0,0.285);
  \draw[gray!50,dashed,thin] (0,\yR) -- (0.5,\yR);
  \draw[gray!50,dashed,thin] (\dd,0) -- (\dd,0.25);
  \draw[gray!50,dashed,thin] (\xb,0) -- (\xb,0.25);
  \draw[thick] (0,0) -- (\dd,\dhalf);
  \draw[thick] (\dd,\yR) -- (\xb,0.25);
  \draw[thick] (\xb,\dhalf) -- (0.5,\yR);
  \node[clo] at (0,0) {};        \node[ope] at (\dd,\dhalf) {};
  \node[clo] at (\dd,\yR) {};    \node[ope] at (\xb,0.25) {};
  \node[clo] at (\xb,\dhalf) {}; \node[ope] at (0.5,\yR) {};
  \foreach \x/\l in {\dd/$\delta$,\xb/$\delta+\frac13$,0.5/$\frac12$}
    \draw (\x,0) -- (\x,-0.01) node[below,font=\scriptsize] {\l};
  \node[below left,font=\scriptsize,inner sep=1pt] at (-0.004,-0.002) {$0$};
  \foreach \y/\l in {\yR/$\frac56$,0.25/$1$}
    \draw (0,\y) -- (-0.009,\y) node[left,font=\scriptsize] {\l};
  \node[left,font=\scriptsize] at (-0.009,0.006) {$\frac34$};
\end{tikzpicture}
\end{minipage}
\caption{The bijection $\theta_2^{(\delta)}\colon Q\to P_2$.}
\label{fig:theta-delta}
\end{figure}

The middle piece has measure $\tfrac13$ and image $R_3$ of measure $\tfrac16$,
while the outer two have total measure $\tfrac16$ and images
$[\tfrac34,\tfrac34+\tfrac\delta2)$ and $[\tfrac34+\tfrac\delta2,\tfrac56)$,
together $P_2\setminus R_3$ of measure $\tfrac1{12}$. The slope being
$\tfrac12$ throughout, $\mu(\theta_2^{(\delta)}(E))=\tfrac12\mu(E)$ for every
measurable $E\subseteq Q$, so \eqref{eq:conformal-basic} holds and the model is
conformal. The sets $P_i$ and $R_j$ are untouched, so all these models share
the first-level data $V$: the $\KMS_1$ states $\varphi_\delta$ of
Theorem~\ref{thm:existence-asymmetric} satisfy $\Phi(\varphi_\delta)=V$, and
the corresponding states $\psi_\delta=2\varphi_\delta|_A$ of $(A,\sigma^A)$
factor through $A/J_Z$ by Lemma~\ref{lem:vertex}(iii). By
Proposition~\ref{prop:corner-model},
$$
        \pi_X(e)=1_{Q_{11}}\delta_e=1_{[0,1/3)}\delta_e,
        \qquad
        \pi_X(f)=1_{Q_{23}}\delta_e=1_{[\delta,\delta+1/3)}\delta_e ,
$$
hence
$$
        \psi_\delta(ef)=\nu\left([\delta,\tfrac13)\right)
        =2\left(\tfrac13-\delta\right)=\tfrac23-2\delta .
$$
Thus $\psi\mapsto\psi(ef)$ attains every value of $[\tfrac13,\tfrac23]$ on
$R_1\left(\Phi^{-1}(V)\right)$, and the bounds in \eqref{eq:overlap} are sharp.
Consequently $\Phi^{-1}(V)$ has affine dimension at least one, and by
Theorem~\ref{thm:fibres-vertex} the quotient system
$(A/J_Z,\widetilde\sigma^A)$ has more than one $\KMS_1$ state.

Since $\psi\mapsto\psi(ef)$ is affine and weak-$*$ continuous and $R_1$ is an
affine homeomorphism, the Bauer maximum principle applied on the compact convex
set $R_1\left(\Phi^{-1}(V)\right)$ shows that the extreme values are attained
at extreme points; hence $\Phi^{-1}(V)$ contains at least two extremal, and
therefore factorial and mutually disjoint, $\KMS_1$ states. Permuting
$t_1,t_2,t_3$ gives, by Proposition~\ref{prop:universal}, an equivariant
automorphism of $\mathcal O_{2,3}$, and the induced permutation of the columns
of $\mathcal P$ acts transitively on its six vertices; the same conclusion
therefore holds at each of them. In particular $\mathcal O_{2,3}$ carries at
least twelve pairwise disjoint extremal factorial $\KMS_1$ states, twice the
bound of Theorem~\ref{thm:vertices}. 

The states $\psi_\delta$ are separated by the value on $ef$, a product of
the range projections of two corner generators, and the higher-level data alluded to in the introduction may well distinguish further points of the fibre.
\end{example}

Although the quotients $A/J_Z$ appearing in Theorem~\ref{thm:fibres-vertex} are not known in general, they can be identified in the diagonal case.

\begin{proposition}\label{prop:diagonal-fibres}
Let $r\ge2$ and let $\Sigma_r$ be the permutation group on $r$ elements.
The vertices of $\mathcal P$ are the matrices $r^{-1}P_\varpi$, where $P_\varpi$ is
the permutation matrix of $\varpi\in \Sigma_r$, and for each of them the fibre
$\Phi^{-1}(r^{-1}P_\varpi)$ is affinely homeomorphic to
$\Tr\left(C^*(\mathbb F_r)\right)$. 
\end{proposition}

\begin{proof}
For $m=n=r$ the constraints \eqref{eq:transport} read
$\sum_jc_{ij}=\sum_ic_{ij}=r^{-1}$, so $r\,\mathcal P$ is the Birkhoff
polytope of $r\times r$ matrices, whose vertices are the
permutation matrices by the Birkhoff--von~Neumann theorem (see e.g. \cite[Definition 2.5 and Theorem 2.12]{DeLoeraKim2014}).

Fix a permutation $\varpi\in \Sigma_r$ and let $Z=\{(i,j):j\ne\varpi(i)\}$ be the complement of the
support of $V=r^{-1}P_\varpi$. In $A/J_Z$ the relations
\eqref{eq:corner-relations} have a single nonzero term on each side, so the
images $\bar v_i$ of $v_{i\varpi(i)}$ are unitaries generating the quotient, and
the argument of Proposition~\ref{prop:free-group-traces} gives
$A/J_Z\cong C^*(\mathbb F_r)$. Since $\varepsilon=0$, the quotient dynamics is
trivial and $K_1(A/J_Z,\widetilde\sigma^{A})\cong\Tr(C^*(\mathbb F_r))$;
Theorem~\ref{thm:fibres-vertex} now identifies this with
$\Phi^{-1}(V)$.
\end{proof}

\section*{Acknowledgments}
The work of Paulo R. Pinto was funded by FCT/Portugal
and the Recovery and Resilience Plan (PRR) through projects UID/04459/2025 and UID/PRR/04459/2025. The work of Filipe Viseu was partially funded by FCT/Portugal through national funds (PIDDAC) under project UIDB/04459/2025 – IST-ID.

\section*{Data availability}
Data sharing is not applicable to this article, as no datasets were generated or analysed during the current study.

\end{document}